\documentclass[11pt]{article}
\usepackage{amscd}
\usepackage{amsfonts}
\usepackage{amsmath}
\usepackage{amssymb}
\usepackage{amsthm}
\usepackage{bbm}
\usepackage{CJK}
\usepackage{fancyhdr}
\usepackage{graphicx}
\usepackage{indentfirst}
\usepackage{latexsym,bm}
\usepackage{mathrsfs}
\usepackage{xypic}
\usepackage[square, comma, sort&compress, numbers]{natbib}
\usepackage[colorlinks,linkcolor=red, anchorcolor=blue]{hyperref}
\allowdisplaybreaks[4]
\newtheorem{theorem}{Theorem}[section]
\newtheorem{lemma}[theorem]{Lemma}
\newtheorem{definition}[theorem]{Definition}
\newtheorem{proposition}[theorem]{Proposition}
\newtheorem{example}[theorem]{Example}

\newtheorem{corollary}[theorem]{Corollary}

\usepackage[top=1in,bottom=1in,left=1.25in,right=1.25in]{geometry}
\def\<{\langle}
\def\>{\rangle}
\def\a{\alpha}
\def\B{\Bbbk}
\def\b{\beta}
\def\bo{\bowtie}

\def\bu{\bullet}
\def\c{\cdot}

\def\d{\delta}
\def\D{\Delta}

\def\f{\forall}
\def\g{\gamma}

\def\ld{\lambda}
\def\lr{\longrightarrow}
\def\lh{\leftharpoonup }
\def\m{\mapsto}
\def\o{\otimes}

\def\r{\rho}

\def\rh{\rightharpoonup }

\def\si{\sigma}
\def\ti{\times}

\def\tr{\triangleright}
\def\tl{\triangleleft}
\def\t{\tau}

\def\v{\varepsilon}
\def\va{\varsigma}

\def\z{\zeta}

\date{}
\begin{document}
\renewcommand{\baselinestretch}{1.2}
\renewcommand{\arraystretch}{1.0}
\title{\bf The Drinfel'd codouble constuction for monoidal Hom-Hopf algebra}
 \date{}
\author {{\bf Dongdong Yan\footnote {Corresponding author:  Ydd150365@163.com} \quad  Shuanhong Wang}\\
{\small Department of Mathematics, Southeast University}\\
{\small Nanjing, Jiangsu 210096, P. R. of China}}
 \maketitle
\begin{center}
\begin{minipage}{12.cm}

\noindent{\bf Abstract.} Let $(H, \b)$ be a monoidal Hom-Hopf algebra with the bijective antipode $S$, In this paper, we mainly construct the Drinfel'd codouble $T(H)=(H^{op}\o H^{*}, \b\o \b^{*-1})$ and $\widehat{T(H)}=( H^{*}\o H^{op}, \b^{*-1}\o \b)$ in the setting of monoidal Hom-Hopf algebras. Then we prove both $T(H)$ and $\widehat{T(H)}$ are coquasitriangular. Finally, we discuss the relation between Drinfel'd codouble and Heisenberg double in the setting of monoidal Hom-Hopf algebras, which is a generalization of the part results in \cite{L94}.
\\

\\

 \end{minipage}
 \end{center}
 \normalsize\vskip1cm

\section*{Introduction}
One of the most celebrated Hopf constructions is the quantum double of Drinfel'd \cite{D90}, which associates to a Hopf algebra $A$ and a quasitriangular Hopf algebra $D(A)$. Unlike the Hopf algebra axioms themselves, the axioms of a dual quasitriangular Hopf algebra (coquasitriangular Hopf algebras) are not self-dual. Thus the axioms and ways of working with these coquasitriangular Hopf algebras look
somewhat different in fact and so it is surely worthwhile to write them out explicity in this dual form and to study them. On the other hand, the corepresentation category of coquasitriangular Hopf algebras can induce a braided monoidal category which is different from one coming from the representation category of quasitriangular Hopf algebras. It is these ideals which many authors studied these notions (see \cite{L94,R93,W99, W01, W07}).

Hom-structures (Lie algebras, algebras, coalgebras, and Hopf algebras) have been intensively investigated in the literature recently (see \cite{CWZ13,HLS06, H99}), In \cite{CG11}, Caenepeel and Goyvaerts illustrated from the point of view of monoidal categories and introduced monoidal Hom-Hopf algebras. Soon afterwards, many of classical results in Hopf algebra theory can be generalized to the monoidal Hom-Hopf algebras (see \cite{LS14,NLZ19}).

In \cite{LW16}, Lu and Wang developed the quasitriangular structure of the Drinfel'd double on Hom-Hopf algebras. In \cite{ZGW19}, the authors discussed the coquasitriangular structure of the Drinfel'd codouble on Hom-Hopf algebras. Motivated by this observation, we aim to find new coquasitriangular structure in the setting of monoidal Hom-Hopf algebras in this paper. This is achieved by generalizing an existing the Drinfel'd codouble (see \cite{M93}). We will generalize this construction to the monoidal Hom-Hopf algebras. We find a suitable generalization of the notions of Majid's bicrossproduct for this setting and obtain a Drinfel'd codouble of a monoidal Hom-Hopf algebra whose corepresentation is a braided monoidal category. What's more, we establish the relation between Drinfel'd codouble and Heisenberg double in the setting of monoidal Hom-Hopf algebras, generalizing the main result in \cite{L94}.

This paper is organized as follows. In section 1, we will recall some of basic definitions and results of monoidal Hom-Hopf algebras, such as
Hom-category, monoidal Hom-Hopf algebras, and their modules and comodules. Let $(A, \a)$ and $(H, \b)$ be two monoidal Hom-Hopf algebras. In section 2, we will introduce the definition of an another bicrossproduct $(A\bo H, \a\o \b)$ (called right bicrossproduct in this paper) which is slightly different from one in \cite{NLZ19}, and the sufficient and necessary conditions for the right bicrossproduct to be a monoidal Hom-Hopf algebra are given, generalizing the structure of the Majid's bicrossproduct defined in \cite{M90}. In section 3, we will give a class of right bicrossproduct monoidal Hom-Hopf algebras.

In section 4, we will introduce the notion of the matched copair and the double crosscoproduct $(A\o H, \a\o \b)$ for monoidal Hom-Hopf algebras $(A, \a)$ and $(H, \b)$, and prove the matched copair is the sufficient condition for the double crosscoproduct to form a monoidal Hom-Hopf algebra, which generalizes the structure of double crosscoproduct in the setting of Hopf algebra (see \cite{CDMP97}). Finally, we obtain the structure of Drinfel'd codouble $T(H)=(H^{op}\o H^{*}, \b\o \b^{*-1})$ and $\widehat{T(H)}=( H^{*}\o H^{op}, \b^{*-1}\o \b)$ by our new approach. In section 5, we will get the Drinfel'd codouble constructed in previous section has a coquasitriangular structure, and study the relation between Drinfel'd codouble and Heisenberg double motivated by the research in \cite{L94}.
\section{Preliminaries}
\def\theequation{1.\arabic{equation}}
\setcounter{equation} {0}
Throughout this paper, all vector spaces, tensor products, and morphisms are over a fixed field $\B$. We denote the identity map by $id$. For a coalgebra $C$, we shall use the Sweedler's notation $\D(c)=c_{1}\o c_{2}$, for any $c\in C$. We always omit the summation symbols for convenience.

Let $\mathcal{M}_{\B}=(\mathcal{M}_{\B},\o, \B, a, l, r)$ be the monoidal category of $\B$-modules. There is a new monoidal category $\mathcal{H}(\mathcal{M}_{\B})=(\mathcal{H}(\mathcal{M}_{\B}),\o,(\B,id_{\B}),  a, l, r)$: the objects are couple $(M,\mu)$, where $M\in \mathcal{M}_{\B}$ and $\mu\in Aut_{\B}(M)$, and a morphism $f: (M, \mu)\lr (N,\nu)$ is a morphism $f: M\lr N$ such that $\nu\circ f=f\circ \mu$. For any objects $(M,\mu)$ and $(N,\nu)$, the monoidal structure is given by
\begin{align*}
(M, \mu)\o (N,\nu)=(M\o N,\mu\o\nu).
\end{align*}
Generally speaking, all Hom-structures are objects in the monoidal category $\widetilde{\mathcal{H}}(\mathcal{M}_{\B})=(\mathcal{H}(\mathcal{M}_{\B}),\o,(\B,id_{\B}),\widetilde{a},\widetilde{l},\widetilde{r})$ (see \cite{CG11}), where the associativity constraint $\widetilde{a}$ and the unit constraints $\widetilde{l}$ and $\widetilde{r}$ are given by the formula
\begin{align*}
\widetilde{a}_{M,N,L}&=a_{M,N,L}\circ ((\mu\o id_{N})\o \va^{-1} )=(\mu\o (id_{N}\o \va^{-1}))\circ a_{M,N,L},
\\ \widetilde{l}_{M }&=\mu\circ l_{M}=l_{M}\circ (id\o \mu),\quad \widetilde{r}_{M }=\mu\circ r_{M}=r_{M}\circ (\mu\o id),
\end{align*}
for any $(M,\mu), (N,\nu), (L,\va)\in \widetilde{\mathcal{H}}(\mathcal{M}_{\B})$. The category $\widetilde{\mathcal{H}}(\mathcal{M}_{\B})$ is called the Hom-category associated to the monoidal category $\mathcal{M}_{\B}$.

In what follows, we will recall from \cite{CG11} some information about Hom-structures.

A unital monoidal Hom-associative algebra $(A,\a)$ is an object in the Hom-category $\widetilde{\mathcal{H}}(\mathcal{M}_{\B})$, with an element $1_{A}\in A$ and a linear map $m_{A}:A\o A\lr A, a\o b\m ab$ such that for any $a,b,c\in A$,
\begin{align*}
\a(a)(bc)&=(ab)\a(c), \quad a 1_{A}=1_{A} a=\a(a),
\\\a(ab)&=\a(a)\a(b),\quad \a(1_{A})=1_{A}.
\end{align*}
In this paper, the algebras we mainly discussed are this kind of unital monoidal Hom-associative algebras, and in the following we call them the monoidal Hom-algebras without any confusion. Let $(A,\a)$ and $(A',\a')$ be two monoidal Hom-algebras. A monoidal Hom-algebra map $f:(A,\a)\lr (A',\a')$ is a map $f:A\lr A'$ such that for any $a,b\in A$, $\a'\circ f=f\circ \a$, $f(ab)=f(a)f(b)$ and $f(1_{A})=1_{A'}$.

A counital monoidal Hom-coassociative coalgebra $(C,\ld)$ is an object in the Hom-category $\widetilde{\mathcal{H}}(\mathcal{M}_{\B})$, with linear maps $\D_{C}:C\lr C\o C, c\m c_{1}\o c_{2}$ and $\v:C\lr \B$ such that for any $c\in C$,
\begin{align*}
\ld^{-1}(c_{1})\o \D(c_{2})&=\D(c_{1})\o \ld^{-1}(c_{2}),\quad c_{1}\v(c_{2})=\v(c_{1})c_{2}=\ld^{-1}(c),
\\\D(\ld(c))&=\ld(c_{1})\o \ld(c_{2}),\quad \v(\ld(c))=\v(c).
\end{align*}
Analogue to monoidal Hom-algebras, monoidal Hom-coalgebras will be short for counital monoidal Hom-coassociative coalgebras without any confusion. Let $(C,\ld)$ and $(C',\ld')$ be two monoidal Hom-coalgebras. A monoidal Hom-coalgebra map $f:(C,\ld)\lr (C',\ld')$ is a map $f:C\lr C'$ such that $\ld'\circ f=f\circ \ld$, $\D'\circ f=(f\o f)\circ \D$ and $\v'\circ f=\v$.

A monoidal Hom-bialgebra $H=(H,\b,m,1_{H},\D,\v)$is a bialgebra in the Hom-category $\widetilde{\mathcal{H}}(\mathcal{M}_{\B})$, this means that $H=(H,\b,m,1_{H} )$ is a monoidal Hom-algebra and $H=(H,\b, \D,\v)$ is a monoidal Hom-coalgebra such that $\D,\v$ are monoidal Hom-algebra maps, i.e. for any $h,k\in H$,
\begin{align*}
\D(hk)&=\D(h)\D(k),\quad  \D(1_{H})=1_{H}\o 1_{H},
\\\v(hk)&=\v(h)\v(k),\quad \v(1_{H})=1_{\B}.
\end{align*}

A monoidal Hom-Hopf algebra is a monoidal Hom-bialgebra $(H,\b)$ with a morphism (called the antipode) $S:H\lr H$ in $\widetilde{\mathcal{H}}(\mathcal{M}_{\B})$ (i.e. $\b\circ S=S\circ \b$), such that for any $h\in H$,
\begin{align*}
S(h_{1})h _{2}=\v(h)1_{H}=h_{1}S(h_{2}).
\end{align*}

Next we will recall the actions and the coactions over monoidal Hom-algebras and monoidal Hom-coalgebras, respectively.

Let $(A,\a)$ be a monoidal Hom-algebra, A left $(A,\a)$-Hom-module is an object $(M,\mu)$ in $\widetilde{\mathcal{H}}(\mathcal{M}_{\B})$ with a morphism $\psi: A\o M\lr M, a\o m\m a\c m$ such that for any $a,b\in A$ and $m\in M$,
\begin{align*}
\a(a)\c (b\c m)&=(ab)\c \mu(m),\quad 1_{A}\c m=\mu(m),
\\\mu(a\c m)&=\a(a)\c\mu(m).
\end{align*}
Similarly, we can define the right $(A,\a)$-Hom-modules. Let $(M,\mu)$ and $(N,\nu)$ be two left $(A,\a)$-Hom-modules, a morphism $f:M\lr N$ is called left $(A,\a)$-linear if for any $a\in A$, $m\in M$, $f(a\c m)=a\c f(m)$ and $f\circ \mu=\nu\circ f$.

Let $(C,\ld)$ be a monoidal Hom-coalgebra. A right $(C,\ld)$-Hom-comodule is an object $(M,\mu)$ in $\widetilde{\mathcal{H}}(\mathcal{M}_{\B})$ with a linear map $\r_{r}:M\lr M\o C, m\m m_{(0)}\o m_{(1)}$ such that for any $m\in M$,
\begin{align*}
\mu^{-1}(m_{(0)})&\o \D(m_{(1)})=m_{(0)(0)}\o (m_{(0)(1)}\o \ld^{-1}(m_{(1)})),
\\  \r(\mu(m))&=\mu(m_{(0)})\o \ld(m_{(1)}),\quad m_{(0)}\v(m_{(1)})=\mu^{-1}(m).
\end{align*}
Similarly, we can define the left $(C,\ld)$-Hom-comodules. Let $(M,\mu)$ and $(N,\nu)$ be two right $(C,\ld)$-Hom-comodules, a morphism $g:M\lr N$ is called right $(C,\ld)$-colinear if for any $m\in M$, $g(m_{(0)})\o m_{(1)}=g(m)_{(0)}\o g(m)_{(1)}$ and $g\circ \mu=\nu\circ g$.

Let $(H,\b)$ be a monoidal Hom-bialgebra. A monoidal Hom-algebra $(A,\a)$ is called a left $(H,\b)$-Hom-module algebra (see \cite{CWZ13}) if $(A,\a)$ is a left $(H,\b)$-Hom-module together with the action $\psi$ such that for any $a,b \in A$ and $h \in H$,
\begin{align*}
h\c (ab) =(h_{1}\c a)(h_{2}\c b), \quad h\c 1_{A} =\v(h)1_{A}.
\end{align*}
\begin{example}\label{E1.1}
Let $(H,\b)$ be a monoidal Hom-bialgebra. The vector space $H^{*}$ becomes a monoidal Hom-algebra with unit $\v_{H}$, and its Hom-multiplication and structure map defined by:
\begin{align*}
&(u\bu v)(h)=u(h_{1})v(h_{2}),
\\&\d:H^{*}\lr H^{*}, \d(u)(h)=u(\b^{-1}(h)),
\end{align*}
for any $u,v \in H^{*}$ and $h\in H$, where $\d\in Aut (H^{*})$ is bijective and $\d^{-1}=\b^{*}$, the transpose of $\b$. Then for any $i\in \mathbb{Z}$, this monoidal Hom-algebra $H^{*}$ can be organized as an left $(H,\b)$-Hom-module algebra (denoted by $H^{*l}_{i}$) with the action defined as follows:
\begin{align*}
\rh : H\o H^{*} \lr H^{*}, \quad (h\rh u)(h')=u(\b^{-2}(h')\b^{i}(h)),
\end{align*}
for any $u\in H^{*}$ and $h,h'\in H$.
\end{example}
When $(H,\b)$ is a monoidal Hom-Hopf algebra, it is easy to prove that $(H^{*},\b^{*-1})$ is also a monoidal Hom-Hopf algebra with the Hom-multiplication defined as above, and the other structure is given by
\begin{align*}
\D_{H^{*}}(u)(x\o y)&=u_{1}(x)u_{2}(y),\quad 1_{H^{*}}:=\v,\quad \v_{H^{*}}(u):=u(1_{H}),
\\&S_{H^{*}}:=S^{*},\quad \b_{H^{*}}:=\d,
\end{align*}
for any $u\in H^{*}$ and $x,y\in H$. We will also denote $u(h)=\<u,h\>$, for any $u\in H^{*}$ and $h\in H$.

Let $(H,\b)$ be a monoidal Hom-bialgebra. A monoidal Hom-coalgebra $(C,\ld)$ is called a left $(H,\b)$-Hom-comodule coalgebra (see \cite{LS14}) if $(C,\ld)$ is a left $(H,\b)$-Hom-comodule together with the coaction $\r(c)=c_{[-1]}\o c_{[0]}$ such that for any $c\in C$,
\begin{align}
 c_{[-1]}\o \D(c_{[0]})&=c_{1[-1]}c_{2[-1]}\o c_{1[0]}\o c_{2[0]},\label{e1.1}
\\\nonumber c_{[-1]}\v(c_{[0]})&=\v(c)1_{H}.
\end{align}

Let $(H,\b)$ be a monoidal Hom-bialgebra. A monoidal Hom-algebra $(C,\ld)$ is called a left $(H,\b)$-Hom-comodule algebra (see \cite{LS14}) if $(C,\ld)$ is a left $(H,\b)$-Hom-comodule together with the coaction $\r_{l}(c)=c_{[-1]}\o c_{[0]}$ such that for any $c,d\in C$,
\begin{align*}
\r(cd)=c_{[-1]}d_{[-1]}\o c_{[0]}d_{[0]},\quad
\r(1_{C})=1_{H}\o 1_{C}.
\end{align*}
Similarly, we can define right $(H,\b)$-Hom-comodule algebra.

Let $(A,\a)$ be a left $(H,\b)$-Hom-comodule coalgebra. The left Hom-smash coproduct $(A\ti H,\a\o \b)$ of $(A,\a)$ and $(H,\b)$ is defined as follows, for any $a\in A$, $h\in H$:
\\(1) $A\ti H=A\o H$ as a vector space,
\\(2) Hom-comultiplication is defined by
\begin{align*}
\D(a\ti h)=a_{1}\ti a_{2[-1]}\b^{-1}(h_{1})\o \a(a_{2[0]})\ti h_{2}.
\end{align*}
Then $(A\ti H,\a\o \b)$ is a monoidal Hom-coalgebra with the counit $\v_{A}\ti \v_{H}$ (see \cite{LS14}).

Finally, we will recall the bicrossproduct structure of two monoidal Hom-Hopf algebras, which is an important result in \cite{NLZ19}. In this paper we call it left bicrossproduct.
\begin{theorem}\label{T1.1}
Let $(A, \a)$ and $(H, \b)$ be two monoidal Hom-Hopf algebras. Suppose that $(A, \a)$ is a left $(H, \b)$-Hom-module algebra with the action $H\o A\lr A, h\o a\m h\tr a$ and $(H, \b)$ is a right $(A, \a)$-Hom-comodule coalgebra with the coaction $H\lr H\o A, h\m h_{(0)}\o h_{(1)}$. Then $(A\o H, \a\o \b)$ is a monoidal Hom-Hopf algebra under the Hom-smash product and Hom-smash coproduct, that is,
\begin{align*}
(a\o h)(b\o g)=a(h_{1}\tr \a^{-1}(b))\o \b(h_{2})g,
\end{align*}
and
\begin{align*}
\D(a\o h)=a_{1}\o \b(h_{1(0)})\o \a^{-1}(a_{2})h_{1(1)}\o h_{2},
\end{align*}
with the antipode $S:A\o H\lr H\o A$ defined by
\begin{align*}
S(a\o h)=(1_{A}\o S_{H}(h_{(0)}))(S_{A}(\a^{-2}(a)\a^{-1}(h_{(1)}))\o 1_{H}),
\end{align*}
if and only if the following conditions are satisfied:
\begin{align*}
&\D(h\tr a)=\b(h_{1(0)})\tr a_{1}\o \a(h_{1(1)})(\b^{-1}(h_{2})\tr \a^{-1}(a_{2})),
\\&\v_{A}(h\tr a)=\v_{H}(h)\v_{A}(a),
\\&\r(1_{H})=1_{H}\o 1_{A},
\\&h_{2(0)}\o (h_{1}\tr a)\a^{2}(h_{2(1)})=h_{1(0)}\o \a^{2}(h_{1(1)})(h_{2}\tr a),
\\&(hg)_{(0)}\o (hg)_{(1)}=\b(h_{1(0)})g_{(0)}\o \a(h_{1(1)})(\b^{-1}(h_{2})\tr \a^{-1}(g_{(1)})).
\end{align*}
\end{theorem}
\section{The right bicrossproduct structure}
\def\theequation{2.\arabic{equation}}
\setcounter{equation} {0}
In this section, let $(A, \a)$ and $(H, \b)$ be two monoidal Hom-Hopf algebras such that $(H, \b)$ is a right $(A, \a)$-Hom-module algebra and that $(A, \a)$ is a left $(H, \b)$-Hom-comodule coalgebra. We will give another construction of the bicrossproduct structure over monoidal Hom-Hopf algebras, which be regarded as the dual of Theorem \ref{T1.1}.
\begin{definition}
Let $(A, \a)$ be a monoidal Hom-bialgebra. A monoidal Hom-algebra $(H, \b)$ is called a right $(A, \a)$-Hom-module algebra if $(H, \b)$ is a right $(A, \a)$-Hom-module with the linear map $\tl : H\o A\lr H: h\o a\m h\tl a$, such that the followings conditions hold:
\begin{align}
 (hg)&\tl a=(h\tl a_{1})(g\tl a_{2}),\label{e2.01}
\\ 1_{H}&\tl a=\v_{A}(a)1_{H},\label{e2.1}
\end{align}
for any $h,g\in H$ and $a\in A$.
\end{definition}
\begin{example}\label{E2.2}
$(1)$ Let $(A,m_{A},\D_{A})$ be a bialgebra and $(R, m_{R})$ a right $A$-module algebra in the usual sense, with the action given by $\c:R\o A\lr R , x\o a\m x\c a$. Let $\a\in Aut (A)$ be a bialgebra automorphism and $\nu\in Aut (R)$ be a algebra automorphism, such that $\nu (x\c a)=\nu(x)\c \a(a)$, for any $a\in A$ and $x\in R$. Consider the monoidal Hom-bialgebra $A_{\a}=(A,\a\circ m_{A},\D_{A}\circ \a^{-1},\a)$ and the monoidal Hom-algebra $R_{\nu}=(R,\nu\circ m_{R},\nu)$. Then the monoidal Hom-algebra $R_{\nu}$ is a right $(A_{\a},\a)$-Hom-module algebra, with the action defined by $R_{\nu}\o A_{\a}\lr R_{\nu}, x\o a\m x\tl a:= \nu(x\c a)=\nu(x)\c \a(a)$.

$(2)$ We recall from Example \ref{E1.1} that the vector space $H^{*}$ is a monoidal Hom-algebra. Then for any $j\in \mathbb{Z}$, this monoidal Hom-algebra $H^{*}$ can be organized as an right $(H,\b)$-Hom-module algebra (denoted by $H^{*r}_{j}$) with the action defined as follows:
\begin{align*}
\lh : H^{*}\o H \lr H^{*}, \quad (u\lh h)(h')=u(\b^{j}(h)\b^{-2}(h')),
\end{align*}
for any $u\in H^{*}$ and $h,h'\in H$.
\end{example}
\begin{lemma}\label{L2.3}
Let $(A, \a)$ be a monoidal Hom-bialgebra and $(H, \b)$ be a monoidal Hom-algebra. If we suppose that $(H, \b)$ is a right $(A, \a)$-Hom-module algebra, Then $(A\# H, \a\o \b)$ is a monoidal Hom-algebra with the unit $1_{A}\#1_{H}$, its Hom-multiplication is given by:
\begin{align*}
(a\# h)(b\# g)=a \a(b_{1})\# (\b^{-1}(h)\tl b_{2})g,
\end{align*}
for any $a, b\in A$ and $h,g\in H$, where $A\# H=A\o H$ as a vector space.
\end{lemma}
\begin{proof}
Firstly, According to the Hom-multiplication and Eq.$(\ref{e2.1})$, It is easy to show that $(a\# h)(1_{A}\#1_{H})=\a(a)\#\b(h)=(1_{A}\#1_{H})(a\# h)$. Secondly, for any $a, b, c\in A$ and $h,g,f\in H$, we obtain:
\begin{align*}
(\a(a)\# &\b(h))((b\# g)(c\# f))
\\&=(\a(a)\# \b(h))(b \a(c_{1})\# (\b^{-1}(g)\tl c_{2})f)
\\&=\a(a) \a((b \a(c_{1}))_{1})\# (h\tl (b \a(c_{1}))_{2})((\b^{-1}(g)\tl c_{2})f)
\\&=\a(a) \a(b_{1} \a(c_{11}))\# (h\tl (b_{2} \a(c_{12})))((\b^{-1}(g)\tl c_{2})f)
\\&=\a(a) \a(b_{1} \a(c_{11}))\# ((\b^{-1}(h)\tl b_{2})\tl \a^{2}(c_{12}))((\b^{-1}(g)\tl c_{2})f)
\\&=\a(a) \a(b_{1} c_{1})\# ((\b^{-1}(h)\tl b_{2})\tl \a^{2}(c_{21}))((\b^{-1}(g)\tl \a(c_{22}))f)
\\&=\a(a) \a(b_{1} c_{1})\# (((\b^{-2}(h)\tl \a^{-1}(b_{2}))\tl \a(c_{21}))(\b^{-1}(g)\tl \a(c_{22})))\b(f)
\\&=\a(a) \a(b_{1} c_{1})\# (((\b^{-2}(h)\tl \a^{-1}(b_{2}))\b^{-1}(g))\tl \a(c_{2}))\b(f)\quad by ~(\ref{e2.01})
\\&=a \a(b_{1}) \a^{2}(c_{1})\# (\b^{-1}((\b^{-1}(h)\tl b_{2})g)\tl \a(c_{2}))\b(f)
\\&=(a \a(b_{1})\# (\b^{-1}(h)\tl b_{2})g)(\a(c)\# \b(f))
\\&=((a\# h)(b\# g))(\a(c)\# \b(f)),
\end{align*}
so the Hom-associativity holds for $(A\# H, \a\o \b)$. This completes the proof.
\end{proof}
Here we will call $(A\# H, \a\o \b)$ a right Hom-smash product.

In what follows, we will give the construction of the right bicrossproduct $(A\bo H, \a\o \b)$ for monoidal Hom-Hopf algebras $(A, \a)$ and $(H, \b)$, its monoidal Hom-algebra structure is right Hom-smash product $(A\# H, \a\o \b)$ induced by the right action of $(A, \a)$ on $(H, \b)$, and its monoidal Hom-coalgebra structure is left Hom-smash coproduct $(A\ti H, \a\o \b)$ induced by the left coaction of $(H, \b)$ on
$(A, \a)$.
\begin{theorem}\label{T2.4}
Let $(A, \a)$ and $(H, \b)$ be two monoidal Hom-Hopf algebras. Suppose that $(A, \a)$ is a left $(H, \b)$-Hom-module algebra with the action $H\o A\lr H, h\o a\m h\tl a$ and that $(H, \b)$ is a right $(A, \a)$-Hom-comodule coalgebra with the coaction $A\lr H\o A, a\m a_{[-1]}\o a_{[0]}$. Then $(A\bo H, \a\o \b)$ is a monoidal Hom-Hopf algebra with the following structures:
\begin{align*}
(a\bo h)(b\bo g)=a \a(b_{1})\bo (\b^{-1}(h)\tl b_{2})g,
\end{align*}
and
\begin{align*}
\D(a\bo h)=(a_{1}\bo a_{2[-1]}\b^{-1}(h_{1}))\o (\a(a_{2[0]})\bo h_{2}),
\end{align*}
and its antipode $S:A\o H\lr H\o A$ is defined by
\begin{align*}
S(a\o h)=(1_{A}\o S_{H}(\b^{-1}(a_{[-1]})\b^{-2}(h)))(S_{A}(a_{[0]})\o 1_{H}),
\end{align*}
if and only if the following conditions are satisfied:
\begin{align}
&\D(h\tl a)=(\b^{-1}(h_{1})\tl \a^{-1}(a_{1}))\b(a_{2[-1]})\o h_{2}\tl \a(a_{2[0]}),\label{e2.2}
\\&\v_{H}(h\tl a)=\v_{H}(h)\v_{A}(a),\label{e2.3}
\\&(ab)_{[-1]}\o (ab)_{[0]}=(\b^{-1}(a_{[-1]})\tl \a^{-1}(b_{1}))\b(b_{2[-1]})\o a_{[0]}\a(b_{2[0]}),\label{e2.4}
\\&\r(1_{A})=1_{H}\o 1_{A},\label{e2.5}
\\&(h\tl a_{1})\b^{2}(a_{2[-1]})\o a_{2[0]}=\b^{2}(a_{1[-1]})(h\tl a_{2})\o a_{1[0]},\label{e2.6}
\end{align}
for any $a,b\in A$ and $h\in H$.
\end{theorem}
\begin{proof}
The proof is similar to Theorem 2.2 and 2.3 in \cite{NLZ19}.
\end{proof}
\begin{example}Let $(A,m_{A},\D_{A},1_{A},\v_{A},S_{A})$ and $(H,m_{H},\D_{H},1_{H},\v_{H},S_{H})$ be two Hopf algebras, such that $H$ is a right $A$-module algebra and $A$ is a left $H$-comodule coalgebra. Suppose that $\a\in Aut(A)$ and $\b\in Aut(H)$ are the automorphisms of Hopf algebras, and satisfy the following conditions:
\begin{align*}
&\b(h\c a)=\b(h)\c \a(a),
\\&\a(a)_{[-1]}\o \a(a)_{[0]}=\b(a_{[-1]})\o \a(a_{[0]}),
\end{align*}
for any $a\in A$ and $h\in H$. Consider the monoidal Hom-Hopf algebras $A_{\a}=(A,\a\circ m_{A},\D_{A}\circ \a^{-1}, \a, 1_{A},\v_{A}, S_{A})$ and $H_{\b}=(H,\b\circ m_{H},\D_{H}\circ \b^{-1}, \b, 1_{H},\v_{H}, S_{H})$, Then $(H_{\b},\b)$ is a right $(A_{\a},\a)$-Hom-module algebra induced by the action $h\tl a=\b(h)\c \a(a)$ and $(A_{\a},\a)$ is a left $(H_{\b},\b)$-Hom-comodule coalgebra induced by the coaction $a_{(-1)}\o a_{(0)}=\b(a_{[-1]})\o \a(a_{[0]})$ satisfy the hypotheses of Theorem $\ref{T2.4}$. Hence we obtain a right bicrossproduct structure $(A_{\a}\o H_{\b},\a\o \b)$.
\end{example}
\section{A class of monoidal Hom-Hopf algebras}
In this section, we will construct a class of monoidal Hom-Hopf algebras as an application of the theory given in the previous section.
\begin{lemma}
Let $(H, \b)$ be a monoidal Hom-Hopf algebra with bijective antipode $S$. If we define the right action of $(H, \b)$ on $(H^{op}, \b)$ by
\begin{align*}
\tl: H^{op}\o H\lr H^{op}: ~x\o h\m  x\tl h=\b^{n+1}S^{-1}(h_{2})(\b^{-1}(x)\b^{n}(h_{1})),  \quad\f x, h\in H,
\end{align*}
and the left coaction of $(H^{op}, \b)$ on $(H, \b)$ by
\begin{align*}
\r: H\lr H^{op}\o H: ~h \m  h_{[-1]}\o h_{[0]}=\b^{n}S^{-1}(h_{22})\b^{n-1}(h_{1})\o \b(h_{21}), \quad\f  h\in H.
\end{align*}
Then $(H^{op}, \b)$ is a right $(H, \b)$-Hom-module algebra, and $(H, \b)$ is a left $(H^{op}, \b)$-Hom-comodule coalgebra.
\end{lemma}
\begin{proof}
The fact that $(H^{op}, \b)$ is a right $(H, \b)$-Hom-module and that $(H, \b)$ is a left $(H^{op}, \b)$-Hom-comodule is obvious and is left to the reader. For any $x, y,h\in H$, we have
\begin{align*}
(xy)\tl h&= \b^{n+1}S^{-1}(h_{2})(\b^{-1}(xy)\b^{n}(h_{1}))
\\&=\b^{n+1}S^{-1}(h_{2})(\b^{-1}(xy)\b^{n}(h_{1}))
\\&=(\b^{n}S^{-1}(h_{2})x)(y\b^{n}(h_{1}))
\\&=[(\b^{n}S^{-1}(h_{22})\b^{-1}(x))(\b^{n+1}(h_{212})\b^{n+1}S^{-1}(h_{211}))](y\b^{n}(h_{1}))
\\&=[(\b^{n}S^{-1}(h_{22})\b^{-1}(x))(\b^{n}(h_{21})\b^{n}S^{-1}(h_{12}))](y\b^{n+1}(h_{11}))
\\&=[\b^{n+1}S^{-1}(h_{22})(\b^{-1}(x)\b^{n}(h_{21}))][\b^{n+1}S^{-1}(h_{12})(\b^{-1}(y)\b^{n}(h_{11}))]
\\&=(x\tl h_{2})(y\tl h_{1}),
\end{align*}
and
\begin{align*}
1_{H}\tl h=\b^{n+1}S^{-1}(h_{2})\b^{n+1}(h_{1})=\v(h)1_{H},
\end{align*}
so $(H^{op}, \b)$ is a right $(H, \b)$-Hom-module algebra. On the other hand, for any $h\in H$, we get
\begin{align*}
h_{[-1]}&\o h_{[0]1}\o h_{[0]2}
\\&=\b^{n}S^{-1}(h_{22})\b^{n-1}(h_{1})\o \b(h_{211})\o \b(h_{212})
\\&=\b^{n}S^{-1}(h_{22})\b^{n}(h_{11}) \o h_{12}\o h_{21}
\\&=(\b^{n-1}S^{-1}(h_{22})(\b^{n-1}(h_{122})\b^{n-1}S^{-1}(h_{121})))\b^{n+1}(h_{111}) \o \b(h_{112})\o h_{21}
\\&=(\b^{n-1}S^{-1}(h_{22})\b^{n}(h_{122}))(\b^{n}S^{-1}(h_{121})\b^{n}(h_{111}) )\o \b(h_{112})\o h_{21}
\\&=(\b^{n}S^{-1}(h_{222})\b^{n-1}(h_{21}))(\b^{n}S^{-1}(h_{122})\b^{n-1}(h_{11}) )\o \b(h_{121})\o \b(h_{221})
\\&=h_{2[-1]}h_{1[-1]}\o h_{1[0]}\o h_{2[0]},
\end{align*}
and
\begin{align*}
\b^{n}S^{-1}(h_{22})\b^{n-1}(h_{1})\v(\b(h_{21}))=\v(h)1_{H},
\end{align*}
so $(H, \b)$ is a left $(H^{op}, \b)$-Hom-comodule coalgebra. This completes the proof.
\end{proof}
\begin{theorem}
Let $(H,\b)$ be a monoidal Hom-Hopf algebra with bijective antipode $S$. We assume that $(H^{op}, \b)$ is a right $(H, \b)$-Hom-module algebra, and that $(H, \b)$ is a left $(H^{op}, \b)$-Hom-comodule coalgebra constructed as above. Then we obtain that the right bicorssproduct $(H\bo H^{op},\b\o\b)$ is a monoidal Hom-Hopf algebra.
\end{theorem}
\begin{proof}
Here we only need to check the five equivalent conditions $(\ref{e2.2})\sim(\ref{e2.6})$ in Theorem \ref{T2.4}. The verification of conditions (\ref{e2.3}) and (\ref{e2.5}) is straightforward. For any $x,h\in H$, we have
\begin{align*}
\D(x\tl h)&=(\b^{n+1}S^{-1}(h_{2})(\b^{-1}(x)\b^{n}(h_{1})))_{1}\o (\b^{n+1}S^{-1}(h_{2})(\b^{-1}(x)\b^{n}(h_{1})))_{2}
\\&=\b^{n+1}S^{-1}(h_{22})(\b^{-1}(x_{1})\b^{n}(h_{11}))\o \b^{n+1}S^{-1}(h_{21})(\b^{-1}(x_{2})\b^{n}(h_{12}))
\\&=\b^{n+1}S^{-1}(h_{22})(\b^{-1}(x_{1})\b^{n-1}(h_{1}))\o \b^{n+2}S^{-1}(h_{212})(\b^{-1}(x_{2})\b^{n+1}(h_{211}))
\\&=[\b^{n}S^{-1}(h_{22})(\b^{n}(h_{122})\b^{n}S^{-1}(h_{121}))](\b^{-1}(x_{1})\b^{n}(h_{11}))
\\&\quad \o \b^{n+2}S^{-1}(h_{212})(\b^{-1}(x_{2})\b^{n+1}(h_{211}))
\\&=(\b^{n}S^{-1}(h_{22})\b^{n+1}(h_{122}))(\b^{n+1}S^{-1}(h_{121})(\b^{-2}(x_{1})\b^{n-1}(h_{11})))
\\&\quad \o \b^{n+2}S^{-1}(h_{212})(\b^{-1}(x_{2})\b^{n+1}(h_{211}))
\\&=(\b^{n+1}S^{-1}(h_{222})\b^{n}(h_{21}))(\b^{n}S^{-1}(h_{12})(\b^{-2}(x_{1})\b^{n-1}(h_{11})))
\\&\quad \o \b^{n+3}S^{-1}(h_{2212})(\b^{-1}(x_{2})\b^{n+2}(h_{2211}))
\\&=\b(h_{2[-1]})(\b^{n}S^{-1}(h_{12})(\b^{-2}(x_{1})\b^{n-1}(h_{11})))
\\&\quad \o \b^{n+2}S^{-1}(h_{2[0]2})(\b^{-1}(x_{2})\b^{n+1}(h_{2[0]1}))
\\&=\b(h_{2[-1]})(\b^{-1}(x_{1})\tl \b^{-1}(h_{1}))\o x_{2}\tl \b(h_{2[0]}),
\end{align*}
we get the condition (\ref{e2.2}). And by
\begin{align*}
(hg)_{[-1]}&\o (hg)_{[0]}
\\&=\b^{n}S^{-1}((hg)_{22})\b^{n-1}((hg)_{1})\o \b((hg)_{21})
\\&=\b^{n+1}S^{-1}(g_{22})((\b^{n-1}S^{-1}(h_{22})\b^{n-2}(h_{1}))\b^{n-1}(g_{1}))\o \b(h_{21})\b(g_{21})
\\&=(\b^{n}S^{-1}(g_{22})(\b^{n}(g_{122})\b^{n}S^{-1}(g_{121})))((\b^{n-1}S^{-1}(h_{22})\b^{n-2}(h_{1}))\b^{n}(g_{11}))
\\&\quad \o \b(h_{21})\b(g_{21})
\\&=(\b^{n+1}S^{-1}(g_{222})(\b^{n-1}(g_{21})\b^{n-1}S^{-1}(g_{12})))((\b^{n-1}S^{-1}(h_{22})\b^{n-2}(h_{1}))\b^{n}(g_{11}))
\\&\quad \o \b(h_{21})\b^{2}(g_{221})
\\&=(\b^{n+1}S^{-1}(g_{222})\b^{n}(g_{21}))(\b^{n}S^{-1}(g_{12})((\b^{n-2}S^{-1}(h_{22})\b^{n-3}(h_{1}))\b^{n-1}(g_{11})))
\\&\quad \o \b(h_{21})\b^{2}(g_{221})
\\&=\b(\b^{n}S^{-1}(g_{222})\b^{n-1}(g_{21}))(\b^{n}S^{-1}(g_{12})(\b^{-2}(\b^{n}S^{-1}(h_{22})\b^{n-1}(h_{1}))\b^{n-1}(g_{11})))
\\&\quad \o \b(h_{21})\b(\b(g_{221}))
\\&=\b(g_{2[-1]})(\b^{n}S^{-1}(g_{12})(\b^{-2}(h_{[-1]})\b^{n-1}(g_{11})))\o h_{[0]}\b(g_{2[0]})
\\&=\b(g_{2[-1]})(\b^{-1}(h_{[-1]})\tl \b^{-1}(g_{1}))\o h_{[0]}\b(g_{2[0]}),
\end{align*}
we obtain the condition (\ref{e2.4}). As for the condition (\ref{e2.6}), we have
\begin{align*}
\b^{2}(h_{2[-1]})&(x\tl h_{1})\o h_{2[0]}
\\&=\b^{2}(h_{2[-1]})(\b^{n+1}S^{-1}(h_{12})(\b^{-1}(x)\b^{n}(h_{11})))\o h_{2[0]}
\\&=\b^{2}(\b^{n}S^{-1}(h_{222})\b^{n-1}(h_{21}))(\b^{n+1}S^{-1}(h_{12})(\b^{-1}(x)\b^{n}(h_{11})))\o \b(h_{221})
\\&=(\b^{n+2}S^{-1}(h_{222})\b^{n+1}(h_{21}))(\b^{n+1}S^{-1}(h_{12})(\b^{-1}(x)\b^{n}(h_{11})))\o \b(h_{221})
\\&=(\b^{n+1}S^{-1}(h_{22})\b^{n+2}(h_{122}))(\b^{n+2}S^{-1}(h_{121})(\b^{-1}(x)\b^{n}(h_{11})))\o h_{21}
\\&=(\b^{n+1}S^{-1}(h_{22})(\b^{n+1}(h_{122})\b^{n+1}S^{-1}(h_{121})))(x\b^{n+1}(h_{11}))\o h_{21}
\\&=\b^{n+2}S^{-1}(h_{22})(x\b^{n}(h_{1}))\o h_{21}
\\&=\b^{n+2}S^{-1}(h_{22})((\b^{-1}(x)(\b^{n+1}(h_{2122})\b^{n+1}S^{-1}(h_{2121})))\b^{n}(h_{1}))\o \b(h_{211})
\\&=\b^{n+2}S^{-1}(h_{22})((\b^{-1}(x)(\b^{n-1}(h_{21})\b^{n}S^{-1}(h_{122})))\b^{n+1}(h_{11}))\o \b(h_{121})
\\&=(\b^{n+1}S^{-1}(h_{22})(\b^{-1}(x)\b^{n}(h_{21})))(\b^{n+2}S^{-1}(h_{122})\b^{n+1}(h_{11}))\o \b(h_{121})
\\&=(\b^{n+1}S^{-1}(h_{22})(\b^{-1}(x)\b^{n}(h_{21})))\b^{2}(\b^{n}S^{-1}(h_{122})\b^{n-1}(h_{11}))\o \b(h_{121})
\\&=(\b^{n+1}S^{-1}(h_{22})(\b^{-1}(x)\b^{n}(h_{21})))\b^{2}(h_{1[-1]})\o h_{1[0]}
\\&=(x\tl h_{2})\b^{2}(h_{1[-1]})\o h_{1[0]}.
\end{align*}
This completes the proof.
\end{proof}
According to Theorem \ref{T2.4}, we obtain the right bicrossproduct structure on $H\o H^{op}$ with the Hom-multiplication and the Hom-comultiplication given as follows:
\begin{align*}
&(h\bo x)(g\bo y)=h \b(g_{1})\bo y(\b^{n+1}S^{-1}(g_{22})(\b^{-2}(x)\b^{n}(g_{21}))),
\\&\D(h\bo x)=(h_{1}\bo \b^{-1}(x_{1})(\b^{n}S^{-1}(h_{222})\b^{n-1}(h_{21})))\o (\b^{2}(h_{221})\bo x_{2}),
\end{align*}
for any $x,y,h,g\in H$.

Similarly, recall from Theorem $\ref{T1.1}$, if we define the left action of $(H,\b)$ on $(H^{op},\b)$ by
\begin{align*}
\tr: H\o H^{op}\lr H^{op}: ~h\o x\m  h\tr x=(\b^{n}(h_{2})\b^{-1}(x))\b^{n+1}S^{-1}(h_{1}),  \quad\f x, h\in H,
\end{align*}
and the right coaction of $(H^{op}, \b)$ on $(H, \b)$ by
\begin{align*}
\r: H\lr H\o H^{op}: ~h \m  h_{(0)}\o h_{(1)}= \b(h_{21})\o \b^{n }(h_{22})\b^{n-1}S^{-1}(h_{1}), \quad\f  h\in H.
\end{align*}
It is not hard to prove that $(H^{op}, \b)$ is a left $(H, \b)$-Hom-module algebra, and that $(H, \b)$ is a right $(H^{op}, \b)$-Hom-comodule coalgebra.
\begin{theorem}
Let $(H,\b)$ be a monoidal Hom-Hopf algebra with bijective antipode. We assume that $(H^{op}, \b)$ is a left $(H, \b)$-Hom-module algebra, and that $(H, \b)$ is a right $(H^{op}, \b)$-Hom-comodule coalgebra constructed as above. Then we get that the left bicorssproduct $(H^{op}\bo H,\b\o\b)$ is a monoidal Hom-Hopf algebra with the Hom-multiplication and the Hom-comultiplication given by:
\begin{align*}
&(x\bo h)(y\bo g)=(\b^{n}(h_{12})(\b^{-2}(y)\b^{n+1}S^{-1}(h_{11})))x\o \b(h_{2})g,
\\&\D(x\bo h)=x_{1}\o \b^{2}(h_{121})\o (\b^{n }(h_{122})\b^{n-1}S^{-1}(h_{11}))\a^{-1}(x_{2})\o h_{2},
\end{align*}
for any $x,y,h,g\in H$.
\end{theorem}
\section{Double crosscoproduct for monoidal Hom-Hopf algebras}
\def\theequation{4.\arabic{equation}}
\setcounter{equation} {0}
The purpose of this Section is to construct the double crosscoproduct for monoidal Hom-Hopf algebra, which generalizes the structure of double crosscoproduct over Hopf algebra (see \cite{CDMP97}), and plays an important role in the construction of Drinfel'd codouble.
\begin{definition}\label{D4.1}
Let $(A,\a)$ and $(H,\b)$ be two monoidal Hom-bialgebras. We call $(A,H)$ a matched copair if $(A,\a)$ is a left $(H,\b)$-Hom-comodule algebra induced by
\begin{align*}
\r_{A}: &A\lr H\o A: a\m a_{[-1]}\o a_{[0]} , \quad \f a\in A,
\end{align*}
and $(H,\b)$ is a right $(A,\a)$-Hom-comodule algebra induced by
\begin{align*}
\r_{H}: H\lr H\o A: h\m h_{(0)}\o h_{(1)}, \quad \f h\in H,
\end{align*}
such that the following conditions hold:
\begin{align}
&a_{[-1]}\v_{A}(a_{[0]})=\v_{A}(a)1_{H};\label{e4.1}
\\&a_{[-1]}\o a_{[0]1} \o a_{[0]2} = a_{1[-1]}\b(a_{2[-1](0)})\o \a^{-1}(a_{1[0]})a_{2[-1](1)}\o a_{2[0]};\label{e4.2}
\\&\v_{H}(h_{(0)})h_{(1)}=\v_{H}(h)1_{A};\label{e4.3}
\\&h_{(0)1}\o h_{(0)2}\o h_{(1)}= h_{1(0)} \o h_{1(1)[-1]}\b^{-1}(h_{2(0)})\o \a(h_{1(1)[0]})h_{2(1)};\label{e4.4}
\\&h_{(0)}a_{[-1]}\o h_{(1)}a_{[0]}=a_{[-1]}h_{(0)}\o a_{[0]}h_{(1)}.\label{e4.5}
\end{align}
\end{definition}
\begin{theorem}
Let $(A,H)$ be a matched copair of monoidal Hom-bialgebras $(A,\a)$ and $(H,\b)$. Then $(A\o H,\a\o\b)$ is a monoidal Hom-bialgebra with unit $1_{A}\o 1_{H}$, such that its Hom-multiplication, Hom-comultiplication and counit is defined as follows, for any $a,b\in A$ and $h,g\in H$,
\begin{align*}
&(a\o h)(b\o g)=ab\o hg,
\\&\D(a\o h)=a_{1}\o a_{2[-1]}h_{1(0)}\o a_{2[0]}h_{1(1)}\o h_{2},
\\&\v(a\o h)=\v_{A}(a)\o \v_{H}(h).
\end{align*}
In this case, we call $A\o H$ the double crosscoproduct of monoidal Hom-bialgebras $(A,\a)$ and $(H,\b)$. Moreover, if $(A,\a)$ and $(H,\b)$ are two monoidal Hom-Hopf algebras with the antipodes $S_{A}$ and $S_{H}$, then $(A\o H,\a\o\b)$ is also a monoidal Hom-Hopf algebra with the antipode defined by:
\begin{align*}
S(a\o h)=S_{A}(a_{[0]}h_{(1)})\o S_{H}(a_{[-1]}h_{(0)}).
\end{align*}
for any $a\in A$ and $h\in H$.
\end{theorem}
\begin{proof}
It is easy to check the fact that $A\o H$ is a monoidal Hom-algebra. For any $a\in A$ and $h\in H$, we compute
\begin{align*}
(\a^{-1}\o&\b^{-1}\o \D)\D(a\o h)
\\&=(\a^{-1}\o\b^{-1}\o \D)(a_{1}\o a_{2[-1]}h_{1(0)}\o a_{2[0]}h_{1(1)}\o h_{2})
\\&=\a^{-1}(a_{1})\o\b^{-1} (a_{2[-1]}h_{1(0)})\o (a_{2[0]}h_{1(1)})_{1}\o (a_{2[0]}h_{1(1)})_{2[-1]}(h_{2})_{1(0)}
\\&\quad\o (a_{2[0]}h_{1(1)})_{2[0]}(h_{2})_{1(1)}\o (h_{2})_{2}
\\& =\a^{-1}(a_{1})\o\b^{-1} (a_{2[-1]}h_{1(0)})\o a_{2[0]1}h_{1(1)1}\o (a_{2[0]2[-1]}h_{1(1)2[-1]})h_{21(0)}
\\&\quad\o (a_{2[0]2[0]}h_{1(1)2[0]})h_{21(1)}\o h_{22}
\\& =\a^{-1}(a_{1})\o\b^{-1} ((a_{21[-1]}\b(a_{22[-1](0)}))h_{1(0)})\o (\a^{-1}(a_{21[0]})a_{22[-1](1)})h_{1(1)1}
\\&\quad\o (a_{22[0][-1]}h_{1(1)2[-1]})h_{21(0)} \o (a_{22[0][0]}h_{1(1)2[0]})h_{21(1)}\o h_{22}\quad by~(\ref{e4.2})
\\& =\a^{-1}(a_{1})\o\b^{-1} ((a_{21[-1]}\b(a_{22[-1](0)}))\b(h_{1(0)(0)}))\o (\a^{-1}(a_{21[0]})a_{22[-1](1)})h_{1(0)(1)}
\\&\quad\o (a_{22[0][-1]}\b^{-1}(h_{1(1)[-1]}))h_{21(0)}\o (a_{22[0][0]}\a^{-1}(h_{1(1)[0]}))h_{21(1)}\o h_{22}
\\&=\a^{-1}(a_{1})\o\b^{-1} (a_{21[-1]}\b(a_{22[-1](0)}))\b(h_{11(0)(0)}\o (\a^{-1}(a_{21[0]})a_{22[-1](1)})\a(h_{11(0)(1)})
\\&\quad\o (a_{22[0][-1]}h_{11(1)[-1]})h_{12(0)}\o (a_{22[0][0]}h_{11(1)[0]})h_{12(1)}\o \b^{-1}(h_{2})
\\&=\a^{-1}(a_{1})\o a_{21[-1]}(a_{22[-1](0)}h_{11(0)(0)})\o a_{21[0]}(a_{22[-1](1)}h_{11(0)(1)})
\\&\quad\o \b(a_{22[0][-1]})(h_{11(1)[-1]}\b^{-1}(h_{12(0)}))\o (a_{22[0][0]}h_{11(1)[0]})h_{12(1)}\o \b^{-1}(h_{2}).
\end{align*}
On the other hand, we have
\begin{align*}
(\D\o &\a^{-1}\o\b^{-1})\D(a\o h)
\\ &=(\D\o \a^{-1}\o\b^{-1})(a_{1}\o a_{2[-1]}h_{1(0)}\o a_{2[0]}h_{1(1)}\o h_{2})
\\ &=(a_{1})_{1}\o (a_{1})_{2[-1]}(a_{2[-1]}h_{1(0)})_{1(0)}\o (a_{1})_{2[0]}(a_{2[-1]}h_{1(0)})_{1(1)}\o (a_{2[-1]}h_{1(0)})_{2}
\\&\quad\o \a^{-1}(a_{2[0]}h_{1(1)})\o \b^{-1}(h_{2})
\\& =a_{11}\o a_{12[-1]}(a_{2[-1]1(0)}h_{1(0)1(0)})\o a_{12[0]}(a_{2[-1]1(1)}h_{1(0)1(1)})\o a_{2[-1]2}h_{1(0)2}
\\&\quad\o \a^{-1}(a_{2[0]}h_{1(1)})\o \b^{-1}(h_{2})
\\& =a_{11}\o a_{12[-1]}(a_{2[-1]1(0)}h_{11(0)(0)})\o a_{12[0]}(a_{2[-1]1(1)}h_{11(0)(1)})
\\&\quad\o a_{2[-1]2}(h_{11(1)[-1]}\b^{-1}(h_{12(0)})) \o \a^{-1}(a_{2[0]}(\a(h_{11(1)(0)})h_{12(1)}))\o \b^{-1}(h_{2})\quad by~ (\ref{e4.4})
\\& =a_{11}\o a_{12[-1]}(\b^{-1}(a_{2[-1](0)})h_{11(0)(0)})\o a_{12[0]}(\a^{-1}(a_{2[-1](1)})h_{11(0)(1)})
\\&\quad\o a_{2[0][-1]}(h_{11(1)[-1]}\b^{-1}(h_{12(0)}))\o \a^{-1}(\a(a_{2[0][0]})(\a(h_{11(1)(0)})h_{12(1)}))\o \b^{-1}(h_{2})
\\&=\a^{-1}(a_{1})\o a_{21[-1]}(a_{22[-1](0)}h_{11(0)(0)})\o a_{21[0]}(a_{22[-1](1)}h_{11(0)(1)})
\\&\quad\o \b(a_{22[0][-1]})(h_{11(1)[-1]}\b^{-1}(h_{12(0)}))\o \a(a_{22[0][0]})(h_{11(1)(0)}\a^{-1}(h_{12(1)}))\o \b^{-1}(h_{2})
\\&=\a^{-1}(a_{1})\o a_{21[-1]}(a_{22[-1](0)}h_{11(0)(0)})\o a_{21[0]}(a_{22[-1](1)}h_{11(0)(1)})
\\&\quad\o \b(a_{22[0][-1]})(h_{11(1)[-1]}\b^{-1}(h_{12(0)}))\o (a_{22[0][0]}h_{11(1)(0)})h_{12(1)}\o \b^{-1}(h_{2}),
\end{align*}
it is obvious to find that the above two equations are equal. The proof that $\v$ is a counit map is straightforward. Therefore $A\o H$ is a monoidal Hom-coalgebra. Next, we prove that $\D$ is an algebra morphism, for any $a,b\in A$ and $h,g\in H$,  we have
\begin{align*}
\D((&a\o h)(b\o g))=\D(ab\o hg)
\\&=(ab)_{1}\o (ab)_{2[-1]}(hg)_{1(0)}\o (ab)_{2[0]}(hg)_{1(1)}\o (hg)_{2}
\\&=a_{1}b_{1}\o (a_{2[-1]}b_{2[-1]})(h_{1(0)}g_{1(0)})\o (a_{2[0]}b_{2[0]})(h_{1(1)}g_{1(1)})\o h_{2}g_{2}
\\&=a_{1}b_{1}\o [a_{2[-1]}\b^{-1}(b_{2[-1]}h_{1(0)})]\b(g_{1(0)})\o [a_{2[0]}\b^{-1}(b_{2[0]}h_{1(1)})]\b(g_{1(1)})\o h_{2}g_{2}
\\&=a_{1}b_{1}\o [a_{2[-1]}\b^{-1}(h_{1(0)}b_{2[-1]})]\b(g_{1(0)})\o [a_{2[0]}\b^{-1}(h_{1(1)}b_{2[0]})]\b(g_{1(1)})\o h_{2}g_{2}\quad by~(\ref{e4.5})
\\&=a_{1}b_{1}\o (a_{2[-1]}h_{1(0)})(b_{2[-1]}g_{1(0)})\o (a_{2[0]}h_{1(1)})(b_{2[0]}g_{1(1)})\o h_{2}g_{2}
\\&=(a_{1}\o a_{2[-1]}h_{1(0)}\o a_{2[0]}h_{1(1)}\o h_{2})(b_{1}\o b_{2[-1]}g_{1(0)}\o b_{2[0]}g_{1(1)}\o g_{2})
\\&=\D(a\o h)\D(b\o g).
\end{align*}
It is easy to compute that $\v$ is also an algebra morphism. Let $(A,\a)$ and $(H,\b)$ be two monoidal Hom-Hopf algebras, we need to prove that $A\o H$ is a monoidal Hom-Hopf algebra, here we only give the proof that $S$ is a left convolution inverse of $id$, the proof that $S$ is a right convolution inverse of $id$ is similar. For any $a\in A$ and $h\in H$, we have
\begin{align*}
S((a\o &h)_{1})(a\o h)_{2}
\\&=S(a_{1}\o a_{2[-1]}h_{1(0)})(a_{2[0]}h_{1(1)}\o h_{2})
\\&=[S_{A}(a_{1[0]}(a_{2[-1]}h_{1(0)})_{(1)})\o S_{H}(a_{1[-1]}(a_{2[-1]}h_{1(0)})_{(0)})](a_{2[0]}h_{1(1)}\o h_{2})
\\&=[S_{A}(a_{1[0]}(a_{2[-1](1)}h_{1(0)(1)}))\o S_{H}(a_{1[-1]}(a_{2[-1](0)}h_{1(0)(0)}))](a_{2[0]}h_{1(1)}\o h_{2})
\\&=S_{A}(a_{1[0]}(a_{2[-1](1)}h_{1(0)(1)}))(a_{2[0]}h_{1(1)})\o S_{H}(a_{1[-1]}(a_{2[-1](0)}h_{1(0)(0)}))h_{2}
\\&=S_{A}((\a^{-1}(a_{1[0]})a_{2[-1](1)})\a(h_{1(0)(1)}))(a_{2[0]}h_{1(1)})
\\&\quad\o S_{H}(\b^{-1}(a_{1[-1]}\b(a_{2[-1](0)}))\b(h_{1(0)(0)}))h_{2}
\\&=S_{A}(a_{[0]1}\a(h_{1(0)(1)}))(a_{[0]2}h_{1(1)})\o S_{H}(\b^{-1}(a_{[-1]})\b(h_{1(0)(0)}))h_{2} \quad by~(\ref{e4.2})
\\&=[S\a(h_{1(0)(1)})(S_{A}\a^{-1}(a_{[0]1})\a^{-1}(a_{[0]2}))]\a(h_{1(1)})\o S_{H}(\b^{-1}(a_{[-1]})\b(h_{1(0)(0)}))h_{2}
\\&=\v_{A}(a)S\a^{2}(h_{1(0)(1)})\a(h_{1(1)})\o S_{H}(\b^{2}(h_{1(0)(0)}))h_{2} \quad by~(\ref{e4.1})
\\&=\v_{A}(a)S\a^{2}(h_{1(1)1})\a^{2}(h_{1(1)2})\o S_{H}(\b(h_{1(0)}))h_{2}
\\&=\v_{A}(a)1_{A}\o S_{H}(h_{1})h_{2}
\\&=\v(a\o h)1_{A}\o 1_{H}
\end{align*}
This completes the proof.
\end{proof}
\begin{proposition}\label{P4.4}
Let $(A\bo H, \a\o \b)$ be a right bicrossproduct of monoidal Hom-Hopf algebras $(A,\a)$ and $(H,\b)$. We define $\r_{1}:H\lr A^{*}\o H$ via
\begin{align*}
\r_{1}(h)=h_{[-1]}\o h_{[0]}=\xi^{s}\o \b^{-2}(h)\tl \a^{-1}(\xi_{s}), \quad \f ~h\in H,
\end{align*}
and $\r_{2}:A^{*}\lr A^{*}\o H$ via
\begin{align*}
\r_{2}(p)=p_{(0)}\o p_{(1)}=\<p,\a^{2}(\xi_{s[0]})\>\xi^{s}\o \b(\xi_{s[-1]}), \quad \f ~p\in A^{*},
\end{align*}
where $\{\xi_{s}\}$, $\{\xi^{s}\}$ is a pair of dual bases in $A$ and $A^{*}$. Then $(H, \b)$ and $(A^{*}, \a^{*-1})$ is a matched copair.
\end{proposition}
\begin{proof}
Firstly, we need to prove that $(H, \b)$ is a left $(A^{*}, \a^{*-1})$-Hom-comodule algebra, and that
$(A^{*}, \a^{*-1})$ is a right $(H, \b)$-Hom-comodule algebra. It is easy to check that $(H, \b)$ is a left $(A^{*}, \a^{*-1})$-Hom-comodule. For any $h, g\in H$ and $a\in A$, we have
\begin{align*}
((hg)_{[-1]}\o (hg)_{[0]}) (a)&=(\xi^{s}\o \b^{-2}(hg)\tl\a^{-1}( \xi_{s}))(a)
\\&=\b^{-2}(hg)\tl\a^{-1}(a)
\\&=( \b^{-2}(h )\tl \a^{-1}(a_{1}))(\b^{-2}( g)\tl \a^{-1}(a_{2}))
\\&=\<\xi^{s}\bu\xi^{t}, a\>( \b^{-2}(h)\tl \a^{-1}(\xi_{s}))(\b^{-2}(g)\tl \a^{-1}(\xi_{t}))
\\&=(\xi^{s}\bu\xi^{t}\o( \b^{-2}(h)\tl \a^{-1}(\xi_{s}))(\b^{-2}(g)\tl \a^{-1}(\xi_{t})))(a)
\\&=(h_{[-1]}g_{[-1]}\o h_{[0]}g_{[0]})(a),
\end{align*}
and $\r_{1}(1_{H})(a)=(\xi^{s}\o 1_{H}\tl \xi_{s})(a)=1_{H}\tl a=\v_{A}(a)1_{H}=(1_{A^{*}}\o 1_{H})(a)$. Therefore $(H, \b)$ is a left $(A^{*}, \a^{*-1})$-Hom-comodule algebra. On the other hand, let $\g^{-1}= \a^{*}$, for any $p\in A^{*}$, $a\in A$ and $u,v\in H^{*}$, we have
\begin{align*}
\<\g^{-1}(p_{(0)})&\o p_{(1)1}\o p_{(1)2},a\o u\o v\>
\\&=\<\g^{-1}(p_{(0)}),a\>\<u, p_{(1)1}\>\<v,p_{(1)2}\>
\\&=\<p_{(0)},\a(a)\>\<u\bu v, p_{(1)}\>
\\&=\<p,\a^{3}(a_{[0]})\>\<u\bu v, \b^{2}(a_{[-1]})\>
\\&=\<p,\a^{3}(a_{[0]})\>\<u, \b^{2}(a_{[-1]1})\>\<v, \b^{2}(a_{[-1]2})\>
\\&=\<p,\a^{4}(a_{[0][0]})\>\<u,\b(a_{[-1]})\>\<\d(v),\b^{3}(a_{[0][-1]})\>
\\&=\<p_{(0)},\a^{2}(a_{[0]})\>\<u,\b(a_{[-1]})\>\<\d(v),p_{(1)}\>
\\&=\<p_{(0)(0)},a\>\<u, p_{(0)(1)}\>\<v,\b^{-1}(p_{(1)})\>
\\&=\<p_{(0)(0)}\o p_{(0)(1)}\o \b^{-1}(p_{(1)}),a\o u\o v\>,
\end{align*}
and
\begin{align*}
\<\r_{2}(\g(p)),a\o u\>&=\<(\g(p))_{(0)}, a\>\<u,(\g(p))_{(1)}\>
\\&=\<\g(p), \a^{2}(a_{[0]})\>\<u, \b(a_{[-1]})\>
\\&=\<\g(p_{(0)}), a\>\<u,\g(p_{(1)})\>=\<\g(p_{(0)})\o \g(p_{(0)}),a\o u\>.
\end{align*}
it is simple to verify that $\v_{A^{*}}(p_{(1)})p_{(0)}=\g^{-1}(p)$.
So we obtain that $(A^{*}, \a^{*-1})$ is a right $(H, \b)$-Hom-comodule. Next, for any $p,q\in A^{*}$ and $a\in A$, we have
\begin{align*}
\r_{2}(p\bu q)(a)&=\<(p\bu q)_{(0)},a\>(p\bu q)_{(1)}
\\&=\<p\bu q, \a^{2}(a_{[0]})\>\b(a_{[-1]})
\\&=\<p, \a^{2}(a_{[0]1})\>\<q, \a^{2}(a_{[0]2})\>\b(a_{[-1]})
\\&=\<p,\a^{2}(a_{1[0]})\>\<q,\a^{2}(a_{2[0]})\>\b(a_{1[-1]})\b(a_{2[-1]})\quad by~(\ref{e1.1})
\\&=\<p_{(0)},a_{1}\>\<q_{(0)},a_{2}\>p_{(1)}q_{(1)}
\\&=(p_{(0)}\bu q_{(0)}\o p_{(1)} q_{(1)})(a),
\end{align*}
and it is easy to prove that $\r_{2}(1_{A^{*}})=1_{A^{*}}\o 1_{H}$. Therefore $(A^{*}, \a^{*-1})$ is a right $(H, \b)$-Hom-comodule algebra. Finally, we check the relations $(\ref{e4.1})\sim (\ref{e4.5})$ given in Definition \ref{D4.1}, Here we only prove the relations (\ref{e4.2}), (\ref{e4.4}) and (\ref{e4.5}). For any $p\in A^{*}$, $a\in A$, $h\in H$ and $u\in H^{*}$, we have
\begin{align*}
(h_{[-1]}&\o h_{[0]1} \o h_{[0]2})(a)
\\& =(\xi^{s}\o (\b^{-2}(h)\tl \a^{-1}(\xi_{s}))_{1} \o (\b^{-2}(h)\tl \a^{-1}(\xi_{s}))_{2})(a)
\\& =(\b^{-2}(h)\tl \a^{-1}(a))_{1} \o (\b^{-2}(h)\tl \a^{-1}(a))_{2}
\\& =(\b^{-1}((\b^{-2}(h))_{1})\tl \a^{-1}((\a^{-1}(a))_{1}))\b((\a^{-1}(a))_{2[-1]})\o (\b^{-2}(h))_{2}\tl \a((\a^{-1}(a))_{2[0]})
\\&=(\b^{-3}(h_{1})\tl \a^{-2}(a_{1}))a_{2[-1]}\o \b^{-2}(h_{2})\tl a_{2[0]}
\\&= \<\xi^{s}, a_{1}\>\<\xi^{t},\a(a_{2[0]})\>(\b^{-3}(h_{1})\tl \a^{-2}(\xi_{s}))a_{2[-1]}\o \b^{-2}(h_{2})\tl \a^{-1}(\xi_{t})
\\&= \<\xi^{s}, a_{1}\>\<\xi^{t}_{(0)},\a^{-1}(a_{2})\>(\b^{-3}(h_{1})\tl \a^{-2}(\xi_{s}))\xi^{t}_{(1)}\o \b^{-2}(h_{2})\tl \a^{-1}(\xi_{t})
\\&= (\xi^{s}\bu\g(\xi^{t}_{(0)})\o \b^{-1}(\b^{-2}(h_{1})\tl \a^{-1}(\xi_{s}))\xi^{t}_{(1)}\o \b^{-2}(h_{2})\tl \a^{-1}(\xi_{t}))(a)
\\&= (h_{1[-1]}\g(h_{2[-1](0)})\o \b^{-1}(h_{1[0]})h_{2[-1](1)}\o h_{2[0]})(a),
\end{align*}
we get the relation (\ref{e4.2}). And by
\begin{align*}
(p_{(0)1}\o p_{(0)2}\o p_{(1)})(a\o b)&=\<p_{(0)1},a\>\<p_{(0)2}, b\>p_{(1)}
\\&=\<p_{(0)},ab\>p_{(1)}=\<p,\a^{2}((ab)_{[0]})\>\b((ab)_{[-1]})
\\&=\<p,\a^{2}(a_{[0]})\a^{3}(b_{2[0]})\>(a_{[-1]}\tl b_{1})\b^{2}(b_{2[-1]})
\\&=\<p_{1} ,\a^{2}(a_{[0]})\>\<p_{2},\a^{3}(b_{2[0]})\>(\b^{-1}(\b(a_{[-1]}))\tl b_{1})\b^{2}(b_{2[-1]})
\\&=\<p_{1(0)} ,a\>\<\xi^{s},b_{1}\>\<p_{2(0)},\a(b_{2})\>\b(\b^{-2}(p_{1(1)})\tl \a^{-1}(\xi_{s}))p_{2(1)}
\\&=\<p_{1(0)} ,a\>\<p_{1(1)[-1]},b_{1}\>\<\g^{-1}(p_{2(0)}),b_{2}\>\b(p_{1(1)[0]})p_{2(1)}
\\&= (p_{1(0)} \o p_{1(1)[-1]}\g^{-1}(p_{2(0)})\o \b(p_{1(1)[0]})p_{2(1)})(a\o b),
\end{align*}
we obtain the relation (\ref{e4.4}). As for the relation (\ref{e4.5}),
\begin{align*}
\<p_{(0)}\bu h_{[-1]}\o p_{(1)}h_{[0]},a\o u\>&=\<p_{(0)},a_{1}\>\<\xi^{s},a_{2}\>\<u_{1},p_{(1)}\>\<u_{2},\b^{-2}(h)\tl \a^{-1}(\xi_{s})\>
\\&=\<p,\a^{2}(a_{1[0]})\>\<u_{1},\b(a_{1[-1]})\>\<u_{2},\b^{-2}(h)\tl \a^{-1}(a_{2})\>
\\&=\<p,\a^{2}(a_{1[0]})\>\<u,\b(a_{1[-1]})(\b^{-2}(h)\tl \a^{-1}(a_{2}))\>
\\&=\<p,\a^{2}(a_{2[0]})\>\<u,(\b^{-2}(h)\tl \a^{-1}(a_{1}))\b(a_{2[-1]})\>\quad by ~(\ref{e2.6})
\\&=\<p,\a^{2}(a_{2[0]})\>\<u_{1},\b^{-2}(h)\tl \a^{-1}(a_{1})\>\<u_{2},\b(a_{2[-1]})\>
\\&=\<\xi^{s},a_{1}\>\<p_{(0)},a_{2}\>\<u_{1},\b^{-2}(h)\tl \a^{-1}(\xi_{s})\>\<u_{2},p_{(1)}\>
\\&=\<h_{[-1]}\bu p_{(0)}\o h_{[0]}p_{(1)},a\o u\>,
\end{align*}
This completes the proof.
\end{proof}
As a conclusion of the above result, we get the double crosscoproduct $(H\o A^{*}, \b \o \a^{*-1})$ with the Hom-multiplication, the Hom-comultiplication given by
\begin{align*}
&(h\o p)(g\o q)=hg\o pq,
\\&\D(h\o p)=h_{1}\o \<p_{1},\a^{2}(\xi_{s[0]})\>\xi^{t}\bu \xi^{s}\o (\b^{-2}(h_{2})\tl \a^{-1}(\xi_{t}))\b(\xi_{s[-1]})\o p_{2},
\end{align*}
for any $h,g\in H$, $p,q\in A^{*}$.
\begin{corollary}\label{C4.5}
Let $(H,\b)$ be a monoidal Hom-Hopf algebra with bijective antipode $S$, then we have the Drinfel'd codouble $T(H)=(H^{op}\bo H^{*},\b\o \b^{*-1})$ with the tensor Hom-multiplication and the Hom-comultiplication given by
\begin{align*}
\D(h\o u)&=h_{1}\o \<u_{1},\b^{2}(e_{s[0]})\>e^{t}\bu e^{s}\o \b(e_{s[-1]})(\b^{-2}(h_{2})\tl \b^{-1}(e_{t}))\o u_{2}
\\&=h_{1}\o \<u_{1},\b^{2}(\b(e_{s21}))\>e^{t}\bu e^{s}\o \b(\b^{n}S^{-1}(e_{s22})\b^{n-1}(e_{s1}))
\\&\quad (\b^{-2}(h_{2})\tl \b^{-1}(e_{t}))\o u_{2}
\\&=h_{1}\o \<u_{1},\b^{2}(\b(e_{s21}))\>e^{t}\bu e^{s}\o \b(\b^{n}S^{-1}(e_{s22})\b^{n-1}(e_{s1}))
\\&\quad [\b^{n}S^{-1}(e_{t2})(\b^{-1}(\b^{-2}(h_{2}))\b^{n-1}(e_{t1}))]\o u_{2}
\\&=h_{1}\o \<u_{1},\b^{3}(e_{s21})\>e^{t}\bu e^{s}\o (\b^{n+1}S^{-1}(e_{s22})\b^{n}(e_{s1}))
\\&\quad [\b^{n}S^{-1}(e_{t2})(\b^{-3}(h_{2})\b^{n-1}(e_{t1}))]\o u_{2}
\\&=h_{1}\o \<u_{1},\b^{3}(e_{s21})\>e^{t}\bu e^{s}\o(\b^{n+1}S^{-1}(e_{s22})(\b^{n-1}(e_{s1})\b^{n-1}S^{-1}(e_{t2})))
\\&\quad (\b^{-2}(h_{2})\b^{n}(e_{t1}))\o u_{2}
\\&=h_{1}\o \<u_{1},\b^{3}(a_{221})\>e^{t}\bu e^{s}\o(\b^{n+1}S^{-1}(a_{222})(\b^{n-1}(a_{21})\b^{n-1}S^{-1}(a_{12})))
\\&\quad (\b^{-2}(h_{2})\b^{n}(a_{11}))\o u_{2}
\\&=h_{1}\o \<u_{1},\b^{2}(a_{21})\>e^{t}\bu e^{s}\o(\b^{n }S^{-1}(a_{22})(\b^{n }(a_{122})\b^{n }S^{-1}(a_{121})))
\\&\quad (\b^{-2}(h_{2})\b^{n}(a_{11}))\o u_{2}
\\&=h_{1}\o \<u_{1},\b^{2}(a_{21})\>e^{t}\bu e^{s}\o
\b^{n+1}S^{-1}(a_{22})(\b^{-2}(h_{2})\b^{n-1}(a_{1}))\o u_{2}
\\&=h_{1}\o \<u_{1},\b^{2}(e_{s1})\>e^{t}\bu e^{s}\o\b^{n+1}S^{-1}(e_{s2})(\b^{-2}(h_{2})\b^{n-1}(e_{t}))\o u_{2}
\\&=h_{1}\o e^{t}\bu(\b^{*2}(u_{1})\bu e^{s})\o\b^{n+1}S^{-1}(e_{s})(\b^{-2}(h_{2})\b^{n-1}(e_{t}))\o u_{2},
\end{align*}
for any $h\in H$ and $u\in H^{*}$, where $\{e_{s}\}$, $\{e^{s}\}$ is a pair of dual bases in $H$ and $H^{*}$.
\end{corollary}
As the dual of Proposition $\ref{P4.4}$, we have the following result.
\begin{proposition}
Let $(A\bo H, \a\o \b)$ be a left bicrossproduct of monoidal Hom-Hopf algebras $(A,\a)$ and $(H,\b)$. We define $\r_{3}:A\lr A\o H^{*}$ via
\begin{align*}
\r_{3}(a)=a_{(0)}\o a_{(1)}=\b^{-1}(e_{s})\tr \a^{-2}(a)\o e^{s}, \quad \f ~a\in A,
\end{align*}
and $\r_{4}:H^{*}\lr  A\o H^{*}$ via
\begin{align*}
\r_{4}(u)=u_{[-1]}\o u_{[0]}=\a (e_{s(1)})\o \<u, \b^{2}(e_{s(0)})\>e^{s}, \quad \f ~u\in H^{*},
\end{align*}
where $\{e_{s}\}$, $\{e^{s}\}$ is a pair of dual bases in $H$ and $H^{*}$. Then $(H^{*}, \b^{*-1})$ and $(A, \a)$ is a matched copair.
\end{proposition}
Therefore, we get the double crosscoproduct $(H^{*}\o A, \b^{*-1}\o \a )$ with the Hom-multiplication and the Hom-comultiplication given by
\begin{align*}
&(u\o a)(v\o b)=uv\o ab,
\\&\D(u\o a)=u_{1}\o \a (e_{s(1)})(\b^{-1}(e_{t})\tr \a^{-2}(a_{1}))\o \<u_{2}, \b^{2}(e_{s(0)})\>e^{s}e^{t}\o a_{2},
\end{align*}
for any $u,v\in H^{*}$ and $a,b\in A$.
\begin{corollary}\label{C4.7}
Let $(H,\b)$ be a monoidal Hom-Hopf algebra with bijective antipode $S$, then we have the Drinfel'd codouble $\widehat{T(H)}=(H^{*}\bo H^{op},\b^{*-1}\o \b)$ with the tensor Hom-multiplication and the Hom-comultiplication given by
\begin{align*}
\D(u\o h)&=u_{1}\o (\b^{-1}(e_{t})\tr \b^{-2}(h_{1}))\b (e_{s(1)})\o \<u_{2}, \b^{2}(e_{s(0)})\>e^{s}\bu e^{t}\o h_{2}
\\&=u_{1}\o [(\b^{n-1}(e_{t2})\b^{-3}(h_{1}))\b^{n }S^{-1}(e_{t1})][\b^{n+1 }(e_{s22})\b^{n }S^{-1}(e_{s1})]
\\&\quad \o \<u_{2}, \b^{3}(e_{s21})\>e^{s}\bu e^{t}(k)\o h_{2}
\\&=u_{1}\o [(\b^{n-1}(k_{22})\b^{-3}(h_{1}))\b^{n }S^{-1}(k_{21})][\b^{n+1 }(k_{122})\b^{n }S^{-1}(k_{11})]
\\&\quad\o \<u_{2}, \b^{3}(k_{121})\>\o h_{2}
\\&=u_{1}\o [(\b^{n-1}(k_{22})\b^{-3}(h_{1}))\b^{n+1 }S^{-1}(k_{212})][\b^{n+1 }(k_{211})\b^{n }S^{-1}(k_{11})]
\\&\quad\o \<u_{2}, \b^{2}(k_{12 })\>\o h_{2}
\\&=u_{1}\o [\b^{n }(k_{22})\b^{-2}(h_{1})][(\b^{n  }S^{-1}(k_{212})\b^{n  }(k_{211}))\b^{n }S^{-1}(k_{11})]
\\&\quad\o \<u_{2}, \b^{2}(k_{12 })\>\o h_{2}
\\&=u_{1}\o [\b^{n-1}(k_{2})\b^{-2}(h_{1})]\b^{n+1}S^{-1}(k_{11})\o \<u_{2}, \b^{2}(k_{12 })\>\o h_{2}
\\&=u_{1}\o (\b^{n-1}(e_{t })\b^{-2}(h_{1}))\b^{n+1}S^{-1}(e_{s})\o  (e^{s}\bu \b^{*2}(u_{2}))\bu e^{t}\o h_{2},
\end{align*}
for any $h\in H$ and $u\in H^{*}$.
\end{corollary}
\section{The Drinfel'd codouble versus the Heisenberg double over monoidal Hom-Hopf algebras}
\def\theequation{5.\arabic{equation}}
\setcounter{equation} {0}
\begin{definition}\label{D5.1}
Let $(H,\b)$ be a monoidal Hom-Hopf algebra. If there exists a convolution invertible bilinear form $\z:H\o H\lr \B$, such that for any $h,g,k\in H$,
\begin{align}
&\z(\b(h),\b(g))=\z(h,g);\label{e5.1}
\\&\z(h_{1},g_{1})g_{2}h_{2}=h_{1}g_{1}\z(h_{2},g_{2});\label{e5.2}
\\&\z(\b^{-1}(h),gk)=\z(h_{1},\b(g))\z(h_{2},\b(k));\label{e5.3}
\\&\z(hg,\b^{-1}(k))=\z(\b(h),k_{2})\z(\b(g),k_{1}),\label{e5.4}
\end{align}
then $\z$ is called a coquasitriangular form of $H$, and $(H,\b,\z)$ is called a coquasitriangular monoidal Hom-Hopf algebra.
\end{definition}
\begin{proposition}\label{P5.2}
Let $(H,\b)$ be a monoidal Hom-Hopf algebra with bijective antipode $S$, then the Drinfel'd codouble $T(H)=(H^{op}\bo H^{*},\b\o \b^{*-1})$ has a coquasitriangular structure given by
\begin{align*}
&\z (h\o p,g\o q)=\<q,\b^{-n}(h)\>p(1_{H})\v(g),
\end{align*}
and its inverse is given by
\begin{align*}
\z^{-1} (h\o p,g\o q)=\<S^{*}(q),\b^{-n}(h)\>p(1_{H})\v(g),
\end{align*}
for any $h,g\in H$ and $p,q \in H^{*}$.
\end{proposition}
\begin{proof}
It is easy to prove that $\z$ is invertible.
We only need to check the four relations in Definition \ref{D5.1}. For any $h\o u, g\o v, k\o w \in T(H)$. Firstly, It is simple to check that
\begin{align*}
&\z(\b(h)\o \b^{*-1 }(u),\b(g)\o \b^{*-1 }(v) )=\z (h\o u,g\o v),
\end{align*}
we get the relation (\ref{e5.1}). And by
\begin{align*}
\z ((h\o& u)_{1},(g\o v)_{1})(g\o v)_{2}(h\o u)_{2}
\\=&\z (h_{1}\o e^{t}\bu(\b^{*2}(u_{1})\bu e^{s}),g_{1}\o e^{j}\bu (\b^{*2}(v_{1})\bu e^{i}))
\\\quad &(\b^{n+1}S^{-1}(e_{i})(\b^{-2}(g_{2})\b^{n-1}(e_{j}))\o v_{2})(\b^{n+1}S^{-1}(e_{s})(\b^{-2}(h_{2})\b^{n-1}(e_{t}))\o u_{2})
\\=&\<e^{j}\bu(\b^{*2}(v_{1})\bu e^{i}),\b^{-n}(h_{1})\>\v(g_{1})(e^{t}\bu(\b^{*2}(u_{1})\bu e^{s}))(1_{H})
\\\quad &[\b^{n+1}S^{-1}(e_{s})(\b^{-2}(h_{2})\b^{n-1}(e_{t}))][\b^{n+1}S^{-1}(e_{i})(\b^{-2}(g_{2})\b^{n-1}(e_{j}))]\o v_{2}u_{2}
\\=&\<v_{1},\b^{-n+2}(h_{121})\>h_{2}[\b S^{-1}(h_{122})(\b^{-3}(g)\b^{ -1}(h_{11}))]\o v_{2}\d^{-1}(u)
\\=&\<v_{1},\b^{-n+1}(h_{21})\>\b^{2}(h_{222})[\b S^{-1}(h_{221})(\b^{-3}(g)\b^{ -2}(h_{1}))]\o v_{2}\d^{-1}(u)
\\=&\<v_{1},\b^{-n+1}(h_{21})\>(\b(h_{222})\b S^{-1}(h_{221}))(\b^{-2}(g)\b^{ -1}(h_{1}))\o v_{2}\d^{-1}(u)
\\=&\<v_{1},\b^{-n}(h_{2})\>\b^{-1}(g)h_{1}\o v_{2}\d^{-1}(u)
\end{align*}
Evaluating the expression above against the tensor $H\o \<\c, x\>$ $(x\in H)$
\begin{align*}
=&\<v_{1},\b^{-n}(h_{2})\>\<v_{2},x_{1}\>\<u,\b(x_{2})>\b^{-1}(g)h_{1}\qquad\qquad\qquad\qquad\qquad\qquad
\\=&\<v,\b^{-n}(h_{2})x_{1}\>\<u,\b(x_{2})>\b^{-1}(g)h_{1}.
\end{align*}
By computing the other side of the relation, we have
\begin{align*}
(h\o& u)_{1}(g\o v)_{1}\z ((h\o u)_{2},(g\o v)_{2})
\\=&(h_{1}\o e^{j}\bu(\b^{*2}(u_{1})\bu e^{i}))(g_{1}\o e^{t}\bu(\b^{*2}(v_{1})\bu e^{s}))
\\\quad &\z (\b^{n+1}S^{-1}(e_{i})(\b^{-2}(h_{2})\b^{n-1}(e_{j}))\o u_{2},\b^{n+1}S^{-1}(e_{s})(\b^{-2}(g_{2})\b^{n-1}(e_{t}))\o v_{2})
\\=&\<v_{2},\b S^{-1}(e_{i})(\b^{-n-2}(h_{2})\b^{ -1}(e_{j}))\>\v(\b^{n+1}S^{-1}(e_{s})(\b^{-2}(g_{2})\b^{n-1}(e_{t})))u_{2}(1_{H})
\\\quad&g_{1}h_{1}\o [e^{j}\bu(\b^{*2}(u_{1})\bu e^{i})][e^{t}\bu (\b^{*2}(v_{1})\bu e^{s})]
\\=&\<v_{2},\b S^{-1}(e_{i})(\b^{-n-2}(h_{2})\b^{ -1}(e_{j}))\>\v(\b^{n+1}S^{-1}(e_{s})(\b^{-2}(g_{2})\b^{n-1}(e_{t})))
\\\quad&g_{1}h_{1}\o [e^{j}\bu(\b^{*3}(u)\bu e^{i})][e^{t}\bu(\b^{*2}(v_{1})\bu e^{s})]
\end{align*}
Evaluating the expression above against the tensor $H\o \<\c, x\>$ $(x\in H)$
\begin{align*}
=&\<v_{2},\b S^{-1}(x_{122})(\b^{-n-2}(h_{2})\b^{ -1}(x_{11}))\>\v(\b^{n+1}S^{-1}(x_{222})(\b^{-2}(g_{2})\b^{n-1}(x_{21})))
\\\quad&g_{1}h_{1}\o u(\b^{3}(x121))v_{1}(\b^{2}(x_{221}))
\\=&\<v,x_{2}[\b S^{-1}(x_{122})(\b^{-n-2}(h_{2})\b^{ -1}(x_{11}))]\>\<u,\b^{3}(x_{121})\>\b^{-1}(g)h_{1}
\\=&\<v,\b^{2}(x_{222})[\b S^{-1}(x_{221})(\b^{-n-2}(h_{2})\b^{ -2}(x_{1}))]\>\<u,\b^{2}(x_{21})\>\b^{-1}(g)h_{1}
\\=&\<v,(\b(x_{222})\b S^{-1}(x_{221}))(\b^{-n-1}(h_{2})\b^{ -1}(x_{1}))\>\<u,\b^{2}(x_{21})\>\b^{-1}(g)h_{1}
\\=&\<v,\b^{-n}(h_{2})x_{1}\>\<u,\b (x_{2})\>\b^{-1}(g)h_{1},
\end{align*}
we easily find that the above two expressions are equal, so the relation (\ref{e5.2}) holds. Finally, we only prove the relation $(\ref{e5.3})$, the relation $(\ref{e5.4})$ can be obtained by a similar computation.
\begin{align*}
\z(\b^{-1}(h)&\o \b^{*}(u),(g\o v)(k\o w))
\\=&\z(\b^{-1}(h)\o \b^{*}(u),kg\o vw)
\\=&\<vw,\b^{-n-1}(h)\>\b^{*}(u)(1_{H})\v(kg)
\\=&\<v,\b^{-n-1}(h_{1})\>\v(g)\<w,\b^{-n-1}(h_{2})\>u(1_{H})\v(k)
\\=&\<v,\b^{-n-1}(h_{1})\>(e^{t}\bu (\b^{*2}(u_{1})\bu e^{s}))(1_{H})\v(g)
\\\quad &\<w,S^{-1}(e_{s})(\b^{-n-3}(h_{2})\b^{ -2}(e_{t}))\>u_{2}(1_{H})\v(k)
\\=&\z(h_{1}\o e^{t}\bu(\b^{*2}(u_{1})\bu e^{s}),\b(g)\o \b^{-1*}(v))
\\\quad &\z(\b^{n+1}S^{-1}(e_{s})(\b^{-2}(h_{2})\b^{n-1}(e_{t}))\o u_{2},\b(k)\o \b^{*-1}(w))
\\=&\z((h\o u)_{1},\b(g)\o \b^{*-1}(v))\z((h\o u)_{2},\b(k)\o \b^{*-1}(w)).
\end{align*}
This completes the proof.
\end{proof}
Similar to the result of Proposition \ref{P5.2}, we have the following conclusion.
\begin{proposition}
Let $(H,\b)$ be a monoidal Hom-Hopf algebra with bijective antipode $S$, then the Drinfel'd codouble $\widehat{T(H)}=(H^{*}\bo H^{op},\b^{*-1}\o \b)$ has a coquasitriangular structure given by
\begin{align*}
&\z (u\o h,v\o k)=\<v,\b^{-n}(h)\>u(1_{H})\v(k),
\end{align*}
and its inverse is given by
\begin{align*}
\z^{-1} (u\o h,v\o k)=\<S^{*}(v),\b^{-n}(h)\>u(1_{H})\v(k),
\end{align*}
for any $h,g\in H$ and $u,v \in H^{*}$.
\end{proposition}
\begin{definition}
Let $(H,\b)$ be a monoidal Hom-Hopf algebra. If the linear map $\si: H\o H\lr \B$ such that the following conditions hold:
\begin{align}
&\si(\b(h), \b(g))=\si(h,g),\label{e5.5}
\\&\si(h_{1},g_{1})\si(h_{2}g_{2}, k)=\si(g_{1},k_{1})\si(h,g_{2}k_{2}),\label{e5.6}
\end{align}
for any $h,g,k\in H$, then $\si$ is called a left monoidal Hom-$2$-cocycle.
\end{definition}
Similarly, if the condition $(\ref{e5.6})$ is replaced by
\begin{align*}
\si(h_{1}g_{1},k)\si(h_{2},g_{2} )=\si(h,g_{1}k_{1})\si(g_{2},k_{2}),
\end{align*}
then $\si$ is called a right monoidal Hom-2-cocycle.

Furthermore, $\si$ is normal if $\si(1,h)=\si (h,1)=\v(h)$.
\begin{proposition}
Let $(H,\b, \z)$ be a coquasitriangular monoidal Hom-Hopf algebra. Then
\\$(1)$ $\si=\z\circ \t :H\o H\lr \B$ is a left monoidal Hom-$2$-cocycle.
\\$(2)$ $\si=\z:H\o H\lr \B$ is a right monoidal Hom-$2$-cocycle.
\end{proposition}
\begin{proof}
(1) It is obvious to obtain Eq. (\ref{e5.5}), we only need to check Eq. (\ref{e5.6}). For any $h,g, k\in H$, we have
\begin{align*}
\si(h_{1},g_{1})\si(h_{2}g_{2}, k)=&\z(g_{1},h_{1})\z(k,h_{2}g_{2})
\\=&\z(k,\z(g_{1},h_{1})h_{2}g_{2})
\\=&\z(k,g_{1}h_{1}\z(g_{2},h_{2}))
\\=&\z(\b(k_{1}),\b(g_{1}))\z(\b(k_{2}),\b(h_{1}))\z(g_{2},h_{2})\quad by ~(\ref{e5.3})
\\=&\z(k_{1},g_{1})\z(k_{2},h_{1})\z(g_{2},h_{2})\quad by ~(\ref{e5.1})
\\=&\z(k_{1},g_{1})\z(g_{2}k_{2},h)
\\=&\si(g_{1},k_{1})\z(h,g_{2}k_{2}).
\end{align*}
(2) Straightforward.

This completes the proof.
\end{proof}

\begin{example}
$(1)$ It follows from Corollary $\ref{C4.5}$ that the Drinfel'd codouble $T(H)$ is the vector space $H^{op}\o H^{*}$ with the Hom-multiplication and the Hom-comultiplication given by
\begin{align*}
&(h\o p)(g\o q)=gh\o pq,
\\&\D(h\o p)=h_{1}\o e^{t}\bu (\b^{*2}(p_{1})\bu e^{s})\o\b^{n+1}S^{-1}(e_{s})(\b^{-2}(h_{2})\b^{n-1}(e_{t}))\o p_{2},
\end{align*}
for any $h,g\in H$ and $p,q\in H^{*}$.

If we define $\si=\z\circ \t:T(H)\o T(H)\lr \B$ by
\begin{align}\label{e5.7}
\si(h\o p,g\o q)=\<p,\b^{-n}(g)\>q(1_{H})\v(h).
\end{align}
Then $\si$ is a left monoidal Hom-$2$-cocycle on $T(H)$.

$(2)$ It follows from Corollary $\ref{C4.7}$ that the Drinfel'd codouble $\widehat{T(H)}$ is the vector space $H^{*}\o H^{op}$ with the Hom-multiplication and the Hom-comultiplication given by
\begin{align*}
&(u\o h)(v\o k)=uv\o kh,
\\&\D(u\o h)=u_{1}\o (\b^{n-1}(e_{t })\b^{-2}(h_{1}))\b^{n+1}S^{-1}(e_{s})\o  (e^{s}\bu (\b^{2})^{*}(u_{2}))\bu e^{t}\o h_{2},
\end{align*}
for any $h,k\in H$ and $u,v\in H^{*}$.

If we define $\si=\z:\widehat{T(H)}\o \widehat{T(H)}\lr \B$ by
\begin{align}\label{e5.8}
\si(u\o h,v\o k)=\<v,\b^{-n}(h)\>u(1_{H})\v(k).
\end{align}
Then $\si$ is a right monoidal Hom-$2$-cocycle on $\widehat{T(H)}$.
\end{example}
\begin{proposition}
Let $(H,\b)$ be a monoidal Hom-Hopf algebra.
\\$(1)$ If $\si$ is a normal left monoidal Hom-$2$-cocycle, define new Hom-multiplication on $H$ as follows
\begin{align*}
h~ \!_{\si}\c g=\si (h_{1},g_{1})\b(h_{2}g_{2}),
\end{align*}
for any $h,g\in H$. Then $(H,\!_{\si}\c, \b)$ is a monoidal Hom-algebra, called the left twist of $H$ (denoted by $\!_{\si}H$).
\\$(2)$ If $\si$ is a normal right monoidal Hom-$2$-cocycle, define new multiplication on $H$ as follows
\begin{align*}
h \c_{\si} g=\b(h_{1}g_{1})\si (h_{2},g_{2}),
\end{align*}
for any $h,g\in H$. Then $(H,\c_{\si}, \b)$ is a monoidal Hom-algebra, called the right twist of $H$ (denoted by $H_{\si}$).
\end{proposition}
\begin{proof}
(1) First of all, for any $h,g,k\in H$, It is easy to check that
\begin{align*}
1~ \!_{\si}\c h=h~ \!_{\si}\c 1=\b(h).
\end{align*}
Next, we have
\begin{align*}
\b(h \c_{\si} g)=\si (h_{1},g_{1})\b^{2}(h_{2}g_{2})=\b(h) \c_{\si} \b(g),
\end{align*}
and
\begin{align*}
\b(h) \c_{\si} (g \c_{\si} k)&=\b(h) \c_{\si} \si (g_{1},k_{1})\b(g_{2}k_{2})
\\&=\si (g_{1},k_{1})\si (\b(h_{1}),\b(g_{21}k_{21}))\b(\b(h_{2})(\b(g_{22}k_{22})))
\\&=\si (g_{1},k_{1})\si (h_{1},g_{21}k_{21})\b^{2}(h_{2})(\b^{2}(g_{22})\b^{2}(k_{22}))
\\&=\si (g_{11},k_{11})\si (h_{1},g_{12}k_{12})(\b (h_{2})\b (g_{2}))\b^{2}(k_{2})
\\&=\si (h_{11},g_{11})\si (h_{12}g_{12},k_{1})(\b (h_{2})\b (g_{2}))\b^{2}(k_{2}) \quad by~(\ref{e5.6})
\\&=\si (h_{1},g_{1})\si (h_{21}g_{21},k_{1})(\b^{2}(h_{22})\b^{2}(g_{22}))\b^{2}(k_{2})
\\&=\si (h_{1},g_{1})\si (\b(h_{21}g_{21}),\b(k_{1}))\b(\b(h_{22}g_{22})\b(k_{2}))
\\&=\si (h_{1},g_{1})\b(h_{2}g_{2})\c_{\si} \b(k)
\\&=(h \c_{\si} g) \c_{\si} \b(k).
\end{align*}
$(2)$ Similar to the proof of the part $(1)$.

This completes the proof.
\end{proof}
\begin{proposition}\label{P5.8}
$(1)$ Let $\si$ is a normal left monoidal Hom-$2$-cocycle on $(H,\b)$. Then the Hom-comultiplication $\D$ of $H$ is an algebra homomorphism from $\!_{\si}H$ to $\!_{\si}H\o H$. Therefore it makes $\!_{\si}H$ into a right $H$-Hom-comodule algebra.
\\$(2)$ If $\si$ is a normal right monoidal Hom-$2$-cocycle on $(H,\b)$. Then the Hom-comultiplication $\D$ of $H$ is an algebra homomorphism from $H_{\si}$ to $H\o H_{\si}$. Therefore it makes $H_{\si}$ into a left $H$-Hom-comodule algebra.
\end{proposition}
\begin{proof}
Straightforward.
\end{proof}
\begin{definition}
The Heisenberg double of a monoidal Hom-Hopf algebra $(H,\b)$, denoted by $\mathcal{H}(H)$, is the right Hom-smash product algebra $H\# H^{*}$ with respect to the right regular action of $H$ on $H^{*}$, i.e., for any $h,g\in H$ and $p,q\in H^{*}$,
\begin{align}\label{e5.9}
(h\# p)(g\# q)=h\b(g_{1})\#(p\circ \b\lh g_{2})\bu q,
\end{align}
where we denote by $H^{*}$ the right $H$-Hom-module algebra $H^{*r}_{-n-1}$ $($ notation as in Example $\ref{E2.2}$ $(2))$, whose $H$-actions $\lh :H^{*}\o H \lr H^{*}, (p\lh h)(g)=p (\b^{-n-1}(h)\b^{-2}(g))$, for any $h,g\in H$ and $p\in H^{*}$.
\end{definition}
Similarly, the Heisenberg double $\mathcal{H}(H^{*})$ of $H^{*}$ is the following left Hom-smash product algebra $H^{*}\#H$:
\begin{align}\label{e5.10}
(u\# h)(v\# k)=u\bu(h_{1}\rh v\circ \b)\# \b(h_{2})k,
\end{align}
where we denote by $H^{*}$ the left $H$-Hom-module algebra $H^{*l}_{-n-1}$ $($ notation as in Example $\ref{E1.1})$, whose $H$-actions $\rh:H\o H^{*}\lr H^{*}, (h\rh u)(k)=u(\b^{-2}(k)\b^{-n-1}(h))$, for any $h,k\in H$ and $u\in H^{*}$.
In what follows, we will give the main result of this section.
\begin{theorem}
Let $(H,\b)$ be a monoidal Hom-Hopf algebra. Then
\\$(1)$ The Heisenberg double $\mathcal{H}(H)$ of $H$ is the left twist of the Drinfel'd codouble $T(H)$ by the left monoidal Hom-$2$-cocycle $\si$ on $T(H)$ given by Eq. $(\ref{e5.7})$.
\\$(2)$ The Heisenberg double $\mathcal{H}(H^{*})$ of $H^{*}$ is the right twist of the Drinfel'd codouble $\widehat{T(H)}$ by the right monoidal Hom-$2$-cocycle $\si$ on $\widehat{T(H)}$ given by Eq. $(\ref{e5.8})$.
\end{theorem}
\begin{proof}
$(1)$ We need to prove that $\!_{\si}T(H)$ and $\mathcal{H}(H)$ share the same Hom-multiplication. In fact, for any $h,g\in H$ and $p,q\in H^{*}$, we have
\begin{align*}
(h\o p&) ~\!_{\si}\c(g\o q)=\si ((h\o p)_{1},(g\o q)_{1})(\b \o \b^{*-1} )[(h\o p)_{2}(g\o q)_{2}]
\\=&\si (h_{1}\o e^{t}\bu(\b^{*2}(p_{1})\bu e^{s}),g_{1}\o e^{j}\bu(\b^{*2}(q_{1})\bu e^{i}))(\b \o \b^{*-1} )
\\\quad &[(\b^{n+1}S^{-1}(e_{s})(\b^{-2}(h_{2})\b^{n-1}(e_{t}))\o p_{2})(\b^{n+1}S^{-1}(e_{i})(\b^{-2}(g_{2})\b^{n-1}(e_{j}))\o q_{2})]
\\=&\<e^{t}\bu(\b^{*2}(p_{1})\bu e^{s}),\b^{-n}(g_{1})\>\v(h_{1})(e^{j}\bu(\b^{*2}(q_{1})\bu e^{i}))(1_{H})
\\\quad &[\b^{n+2}S^{-1}(e_{i})(\b^{-1}(g_{2})\b^{n}(e_{j}))][\b^{ n+2}S^{-1}(e_{s})(\b^{-1}(h_{2})\b^{n}(e_{t}))]\o \b^{*-1}(p_{2}) \b^{*-1} (q_{2})
\\=&\<p_{1},\b^{-n+2}(g_{121})\>\b (g_{2})[\b^{2}S^{-1}(g_{122})(\b^{-2}(h)g_{11})]\o \b^{*-1} (p_{2}) q
\\=&\<p_{1},\b^{-n+1}(g_{21})\>\b^{3}(g_{222})[\b^{2}S^{-1}(g_{221})(\b^{-2}(h)\b^{-1}(g_{1}))]\o \b^{*-1} (p_{2}) q
\\=&\<p_{1},\b^{-n+1}(g_{21})\>(\b^{2}(g_{222})\b^{2}S^{-1}(g_{221}))(\b^{-1}(h)g_{1})\o \b^{*-1} (p_{2}) q
\\=&\<p_{1},\b^{-n}(g_{2})\>(h\b (g_{1}))\o \b^{*-1}(p_{2}) q.
\end{align*}
Comparing with Eq. $(\ref{e5.9})$, we easily find that this is the Hom-multiplication in $\mathcal{H}(H)$.

$(2)$ For any $h,k\in H$ and $u,v\in H^{*}$, we get
\begin{align*}
(u\# &h)\c _{\si}(v\# k)=(\b^{*-1}\o \b)((u\# h)_{1}(v\# k)_{1})\si ((u\# h)_{2},(v\# k)_{2})
\\&=((\b^{*-1}\o \b)((u_{1}\o (\b^{n-1}(e_{t })\b^{-2}(h_{1}))\b^{n+1}S^{-1}(e_{s}))(v_{1}\o (\b^{n-1}(e_{j })\b^{-2}(k_{1}))
\\&\quad \b^{n+1}S^{-1}(e_{i}))) \si ((e^{s}\bu \b^{*2}(u_{2}))\bu e^{t}\o h_{2},(e^{i}\bu (\b^{2})^{*}(v_{2}))\bu e^{j}\o k_{2})
\\&=\b^{*-1}(u_{1})\bu \b^{*-1}(v_{1})\o [(\b^{n }(e_{j })\b^{-1}(k_{1}))\b^{n+2}S^{-1}(e_{i})][(\b^{n }(e_{t })\b^{-1}(h_{1}))\b^{n+2}S^{-1}(e_{s})]
\\&\quad  \<(e^{i}\bu (\b^{*2}(v_{2}))\bu e^{j},\b^{-n}(h_{2})\>((e^{s}\bu \b^{*2}(u_{2}))\bu e^{t})(1_{H})\v(k_{2})
\\&=\<v_{2},\b^{-n+2}(h_{212})\>u\bu v_{1}\circ \b^{-1}\o [(h_{22}\b^{-2}(k))\b^{2}S^{-1}(h_{211})]\b (h_{1})
\\&=\<v_{2},\b^{-n+1}(h_{21 })\>u\bu v_{1}\circ \b^{-1}\o [(h_{22}\b^{-2}(k))\b S^{-1}(h_{12})]\b^{2} (h_{11})
\\&=\<v_{2},\b^{-n+1}(h_{21 })\>u\bu v_{1}\circ \b^{-1}\o (\b(h_{22})\b^{-1}(k))(\b S^{-1}(h_{12})\b  (h_{11}))
\\&=\<v_{2},\b^{-n }(h_{1 })\>u\bu v_{1}\circ \b^{-1}\o \b(h_{2 })k.
\end{align*}
Comparing with Eq. $(\ref{e5.10})$, we easily see that this is the Hom-multiplication in $\mathcal{H}(H^{*})$.
This completes the proof.
\end{proof}
The next corollary follows from Proposition \ref{P5.8}.
\begin{corollary}
$(1)$ The Hom-comultiplication of $T(H)$, considered as a map from $\mathcal{H}(H)$ to $\mathcal{H}(H)\o T(H)$, makes $\mathcal{H}(H)$ into a right $T(H)$-Hom-comodule algebra.
\\$(2)$ The Hom-comultiplication of $\widehat{T(H)}$, considered as a map from $\mathcal{H}(H^{*})$ to $\widehat{T(H)}\o \mathcal{H}(H^{*})$, makes $\mathcal{H}(H^{*})$ into a left $\widehat{T(H)}$-Hom-comodule algebra.
\end{corollary}

\end{document}